\pdfoutput=1
\documentclass[12pt,reqno]{amsart}

\usepackage{amsmath, amssymb, amsthm, mathrsfs}
\usepackage{mathabx}\changenotsign        
\usepackage{dsfont}
\usepackage[shortlabels]{enumitem} 
\usepackage{lmodern}
\usepackage[babel]{microtype}
\usepackage[british]{babel}
\usepackage[utf8]{inputenc}

\usepackage{xcolor}
\usepackage[backref,hypertexnames=false]{hyperref} 
\hypersetup{
	colorlinks,
	linkcolor={red!60!black},
	citecolor={green!60!black},
	urlcolor={blue!60!black}
}

\usepackage[open,openlevel=2,atend]{bookmark}
\usepackage[abbrev,msc-links,backrefs]{amsrefs}
\usepackage{doi}

\renewcommand{\PrintDOI}[1]{\doi{#1}}

\def\ssign{\textsection\nobreak\hspace{1pt plus 0.3pt}}
\makeatletter
\let\origsection=\section 
\def\mysection{\@mystartsection{section}{1}\z@{.7\linespacing\@plus\linespacing}{.5\linespacing}{\normalfont\scshape\centering\ssign}}
\def\section{\@ifstar{\origsection*}{\mysection}}
\def\appendix{\par\c@section\z@ \c@subsection\z@
   \let\sectionname\appendixname
   \let\section=\origsection
   \def\thesection{\@Alph\c@section}}
\def\@mystartsection#1#2#3#4#5#6{\if@noskipsec \leavevmode \fi
 \par \@tempskipa #4\relax
 \@afterindenttrue
 \ifdim \@tempskipa <\z@ \@tempskipa -\@tempskipa \@afterindentfalse\fi
 \if@nobreak \everypar{}\else
     \addpenalty\@secpenalty\addvspace\@tempskipa\fi
 \@dblarg{\@mysect{#1}{#2}{#3}{#4}{#5}{#6}}}
\def\@mysect#1#2#3#4#5#6[#7]#8{\edef\@toclevel{\ifnum#2=\@m 0\else\number#2\fi}\ifnum #2>\c@secnumdepth \let\@secnumber\@empty
  \else \@xp\let\@xp\@secnumber\csname the#1\endcsname\fi
  \@tempskipa #5\relax
  \ifnum #2>\c@secnumdepth
    \let\@svsec\@empty
  \else
    \refstepcounter{#1}\edef\@secnumpunct{\ifdim\@tempskipa>\z@ \@ifnotempty{#8}{\@nx\enspace}\else
        \@ifempty{#8}{.}{\@nx\enspace}\fi
    }\@ifempty{#8}{\ifnum #2=\tw@ \def\@secnumfont{\bfseries}\fi}{}\protected@edef\@svsec{\ifnum#2<\@m
        \@ifundefined{#1name}{}{\ignorespaces\csname #1name\endcsname\space
        }\fi
      \@seccntformat{#1}}\fi
  \ifdim \@tempskipa>\z@ \begingroup #6\relax
    \@hangfrom{\hskip #3\relax\@svsec}{\interlinepenalty\@M #8\par}\endgroup
    \ifnum#2>\@m \else \@tocwrite{#1}{#8}\fi
  \else
  \def\@svsechd{#6\hskip #3\@svsec
    \@ifnotempty{#8}{\ignorespaces#8\unskip
       \@addpunct.}\ifnum#2>\@m \else \@tocwrite{#1}{#8}\fi
  }\fi
  \global\@nobreaktrue
  \@xsect{#5}}
\makeatother

\usepackage{geometry}
\def\alabel{\upshape({\itshape \alph*\,})}

\let\setminus=\smallsetminus
\let\emptyset=\varnothing

\let\lra=\longrightarrow
\let\to=\lra
\def\tand{\ \text{and}\ }
\def\qand{\quad\text{and}\quad}

\newcommand{\HH}{\mathds H}
\newcommand{\NN}{\mathds N}

\makeatletter
\def\moverlay{\mathpalette\mov@rlay}
\def\mov@rlay#1#2{\leavevmode\vtop{   \baselineskip\z@skip \lineskiplimit-\maxdimen
   \ialign{\hfil$\m@th#1##$\hfil\cr#2\crcr}}}
\newcommand{\charfusion}[3][\mathord]{
    #1{\ifx#1\mathop\vphantom{#2}\fi
        \mathpalette\mov@rlay{#2\cr#3}
      }
    \ifx#1\mathop\expandafter\displaylimits\fi}
\makeatother

\newcommand{\dcup}{\charfusion[\mathbin]{\cup}{\cdot}}

\let\discup=\dcup

\usepackage{stackrel}
\usepackage{graphicx}

\numberwithin{equation}{section}

\setlist[description]{font=\normalfont}

\theoremstyle{plain}
\newtheorem{theorem}{Theorem}[section]
\newtheorem{lemma}[theorem]{Lemma}
\newtheorem{claim}[theorem]{Claim}
\newtheorem{proposition}[theorem]{Proposition}
\newtheorem{fact}[theorem]{Fact}

\newtheorem{corollary}[theorem]{Corollary}

\theoremstyle{definition}
\newtheorem{remark}[theorem]{Remark}
\newtheorem{definition}[theorem]{Definition}

\newtheorem{conjecture}[theorem]{Conjecture}

\newcommand{\sm}{\setminus}
\DeclareMathOperator{\Ex}{\mathds{E}}
\renewcommand{\Pr}{\mathds{P}}

\let\eps=\varepsilon
\let\theta=\vartheta
\let\rho=\varrho
\let\phi=\varphi

\usepackage{framed}

\makeatletter
\renewcommand*{\eqref}[1]{\hyperref[{#1}]{\textup{\tagform@{\ref*{#1}}}}}
\makeatother

\let\tilde=\widetilde

\title{Ramsey properties of randomly perturbed hypergraphs}

\author[E.~Aigner-Horev]{Elad Aigner-Horev}
\address{School of Computer Science, Ariel University, Ariel, Israel}
\email{horev@ariel.ac.il}
\email{danhe@ariel.ac.il}

\author[D.~Hefetz]{Dan Hefetz}

\author[M.~Schacht]{Mathias Schacht}
\address {
	Fachbereich Mathematik,
	Universität Hamburg,
	Hamburg, Germany
}
\email{schacht@math.uni-hamburg.de}

\begin{document}

\begin{abstract}
	We study Ramsey properties of randomly perturbed $3$-uniform hypergraphs. For~$t\geq 2$, write $\tilde K^{(3)}_t$ to denote the $3$-uniform {\it expanded} clique hypergraph obtained from the complete graph $K_t$ by expanding each of the edges of the latter with a new additional vertex. For an even integer $t\geq 4$, let~$M$ denote the asymmetric maximal density of the pair $(\tilde K^{(3)}_t,\tilde K^{(3)}_{t/2})$. We prove that adding a set~$F$ of random hyperedges  satisfying $|F|\gg n^{3-1/M}$ to a given $n$-vertex $3$-uniform hypergraph~$H$ with non-vanishing edge density asymptotically almost surely results in a perturbed hypergraph enjoying the Ramsey property for $\tilde K^{(3)}_t$ and two colours. We conjecture that this result is asymptotically best possible with respect to the size of $F$ whenever $t\geq 6$ is even. The key tools of our proof are a new variant of the hypergraph regularity lemma accompanied with a \emph{tuple lemma} providing appropriate control over joint link graphs.
	Our variant combines the so called strong and the weak hypergraph regularity lemmata.
\end{abstract}

\maketitle

\section{Introduction}
\subsection{Ramsey properties of random hypergraphs}
Given a distribution $\mathcal{R}$ over $n$-vertex hypergraphs, as well as an $n$-vertex hypergraph~$H$, referred to as the {\em seed} hypergraph, unions of the form $H \cup R$ with $R \sim \mathcal{R}$ define a distribution over the super-hypergraphs of $H$, denoted by $H \cup \mathcal{R}$. The hypergraphs $H \cup \mathcal{R}$ are referred to as {\em random perturbations} of $H$. The study of the properties of randomly perturbed hypergraphs has received some attention in recent years. Thus far, two dominant strands of results in this avenue have emerged. One strand is the study of the thresholds for the emergence of various spanning and nearly-spanning configurations within such structures (see, e.g.,~\cites{AHhamilton,AHK22b,AHK22a,AHTrees,BTW17,BHKM18,BHKMPP18,BMPP18,DRRS18,HZ18,KKS16,KKS17,MM18}). The second strand pertains to their extremal and Ramsey-type properties (see, e.g.,~\cites{ADHLlarge,ADHLsmall,AHTrees,AHP,DKM21,DMT20,DT19,KST,Powierski19}). Our result lies in the latter vein. We recall the arrow notation
\[
	G\lra(H_1,H_2)\,,
\]
signifying the validity of the asymmetric Ramsey statement that every $2$-colouring of the edges of $G$ yields a monochromatic copy of $H_1$ in the first colour
or a monochromatic copy of $H_2$ in the second colour. Moreover, in the symmetric case when $H_1=H_2=H$
we simply write $G\lra(H)$.

Ramsey properties of randomly perturbed graphs were first investigated by Krivelevich, Sudakov, and Tetali~\cite{KST}.
In that work it was shown that $n^{-2/(t-1)}$ is the threshold for the asymmetric Ramsey property $G \cup \mathds{G}(n,p) \to (K_3,K_t)$,
whenever $G$ is an $n$-vertex graph of edge density $d \in (0,1/2)$ independent of $n$.
The general problem, put forth by Krivelevich~et~al., of determining the threshold for the property
$G \cup \mathds{G}(n,p) \to (K_s, K_t)$, whenever $G$ is dense and $s, t \geq 4$,
was recently (essentially) resolved by Das and Treglown~\cite{DT19}. Those authors showed that $n^{-1/m_2(K_t,K_{\lceil s/2 \rceil})}$ is the threshold for the property $G \cup \mathds{G}(n,p) \to (K_s, K_t)$, when~$G$ is a dense $n$-vertex graph and
$t \geq s \geq 5$, where $m_2(H_1, H_2)$ denotes the asymmetric maximal {$2$-density} of two
graphs $H_1$ and $H_2$ (see Equation~\eqref{eq:asym-density} for the definition). For other values of~$t$ and $s$ we also refer to the work of Das and Treglown~\cite{DT19}*{Theorem~1.7 and Theorem~5(ii)} and for the special case $s = t = 4$ in addition to the work of Powierski~\cite{Powierski19}*{Theorem~1.8}.

The aforementioned Ramsey-type results for randomly perturbed dense graphs are formulated for $2$-colourings only. This restriction is well-justified. Indeed,  suppose that more than two colours are available. The colouring in which the seed is coloured using one colour and the random perturbation is coloured using all the remaining colours, reduces the problem to that of studying the Ramsey property at hand for truly random hypergraphs.

The earlier results~\cites{DT19,KST,Powierski19}, as well as our result, stated in Theorem~\ref{thm:main-Ramsey} below, are affected by and closely related to research on Ramsey properties in random graphs and hypergraphs (see, e.g.,~\cites{CG16,FRS10,GNPSST17,Hyde,KK97,KSS14,LMMS20,LRV92,MSSS09,MNS20,NPSS17,NS16,RR93,RR94,RR95,RR98}). For random graphs, the thresholds for {\it symmetric} Ramsey properties are well-understood due to work of R\"odl and Ruci\'nski~\cites{RR93,RR95}. Minor exceptions for $F$ being a star forest aside, this work asserts that $n^{-1/m_2(F)}$ is the threshold for the property $\mathds{G}(n,p) \to (F)$, where $m_2(F)$ denotes the maximal $2$-density of the given
graph $F$ (see Equation~\eqref{eq:sym-density} for the definition). The $1$-statement of the threshold was extended to random $k$-uniform hypergraphs by Conlon and Gowers~\cite{CG16} and
by Friedgut, R\"odl, and Schacht~\cite{FRS10}. However, a complete characterisation of the exceptional cases is not yet available and for the progress towards the $0$-statement we refer to the work of Nenadov et al.~\cite{NPSS17} and Gugelmann~et~al.~\cite{GNPSST17}.

The thresholds of asymmetric Ramsey properties in random graphs are the subject of the {\it Kohayakawa--Kreuter conjecture}~\cites{KK97}. The $1$-statement stipulated by this conjecture has been fairly recently verified by Mousset, Nenadov, and Samotij~\cite{MNS20} and progress has been made with respect to the corresponding $0$-statement by several researchers~\cites{GNPSST17,Hyde,LMMS20,MSSS09}. Following some progress~\cites{BHH, KSY, MNS20}, the conjecture  was finally fully resolved by Christoph, Martinsson, Steiner, and Wigderson~\cite{CMSW}.

\subsection{Main result}\label{sec:main-results}
We study Ramsey properties of randomly perturbed hypergraphs; stating our results requires preparation.
A hypergraph $H$ is said to be {\em linear} if $|e \cap f| \leq 1$ holds whenever $e$, $f \in E(H)$ are distinct. Amongst the linear hypergraphs, {\it expanded cliques} are of special interest. Given $t \geq 2$ and $k \geq 2$, the {\em $k$-uniformly expanded clique of order $t$}, denoted by $\tilde K^{(k)}_t$, is the $k$-uniform hypergraph with vertex set of size $t + \binom{t}{2} (k-2)$ obtained from the complete graph $K_t$ by expanding every edge of $K_t$ by~$k-2$ new vertices; in particular, $\tilde K^{(2)}_t = K_t$ holds. Expanded cliques have attracted some attention in the literature and related extremal and Ramsy-type questions were addressed by Mubayi~\cite{M06} and by Conlon, Fox, and R\"odl~\cite{CFR17}.

Two natural measures of density, arising in the context of random hypergraphs, are the {\it maximum density} of a $k$-uniform $H=(V,E)$, denoted $m(H)$, and its {\it maximum $k$-density}, denoted $m_k(H)$. The former is given by
\[
	m(H)
	=
	\max \left\{\frac{e(F)}{v(F)}\colon F \subseteq H \tand v(F) \geq 1 \right\}\,
\]
and the latter is defined by
\begin{equation}\label{eq:sym-density}
	m_k(H) = \max\big\{d_k(F)\colon {F \subseteq H}\big\}\,,
	\ \text{where}\
	d_k(F) =
	\begin{cases}
		0,                     & \text{if $e(F) = 0$},             \\
		\frac{1}{k},           & \text{if $e(F) = 1$, $v(F) = k$}, \\
		\frac{e(F)-1}{v(F)-k}, & \text{otherwise.}
	\end{cases}
\end{equation}
It is well known that $n^{-1/m(H)}$ is the threshold for the appearance of $H$ as a subhypergraph in the binomial
random $k$-uniform hypergraph $\HH^{(k)}(n,p)$. For $\HH^{(k)}(n,p)$ to satisfy the Ramsey property for $H$ asymptotically almost surely (hereafter, a.a.s.\ for brevity) it is reasonable to expect that many intermingled copies of $H$ are required; this as to create colour restrictions forcing the Ramsey property for $H$. Indeed, for (hypergraph) cliques it is necessary that many cliques sharing a single hyperedge would appear a.a.s.\ in $\HH^{(k)}(n,p)$. This results in the higher threshold~$n^{-1/m_k(H)}$ being encountered for Ramsey properties.

For asymmetric Ramsey properties, another notion of hypergraph density arises. This notion traces back to the work of Kohayakawa and Kreuter~\cite{KK97}. Given two $k$-uniform hypergraphs~$H_1$ and $H_2$, each with at least one edge and  satisfying
$m_k(H_1) \geq m_k(H_2)$, the {\em asymmetric maximal $k$-density} of $H_1$ and $H_2$ is given by
\begin{equation}\label{eq:asym-density}
	m_k(H_1, H_2)=m_k(H_2, H_1)
	=
	\max \left\{\frac{e(F)}{v(F) - k  +1/m_k(H_2)}\colon F \subseteq H_1 \tand e(F) \geq 1 \right\}.
\end{equation}
The equality $m_k(H,H)=m_k(H)$ is easy to verify.

With the above notation in place, our main contribution can be stated; this can be viewed as a hypergraph extension of the aforementioned results of Das and Treglown~\cites{DT19}. Below we always tacitly assume that $H_n$ and $\HH^{(3)}(n,p)$ share the same vertex set.

\begin{theorem}[Main result]
	\label{thm:main-Ramsey}
	For every $d > 0$ and every even integer $t \geq 4$, there exists a constant $C>0$ such that
	for every sequence of $3$-uniform $n$-vertex hypergraphs $(H_n)_{n\in\NN}$ with $e(H_n) \geq d n^3$
	for every $n\in\NN$, we have
	$$
		\lim_{n \to \infty} \Pr\big(H_n \cup \mathds{H}^{(3)}(n,p) \lra (\tilde K^{(3)}_t)\big)
		=
		1,
	$$
	whenever $p = p(n) \geq Cn^{- 1/M}$ for $M=m_3(\tilde K^{(3)}_t,  \tilde K^{(3)}_{t/2})$.
\end{theorem}

The proof of Theorem~\ref{thm:main-Ramsey} presented here can be adapted for $k$-uniform hypergraphs
and the asymmetric Ramsey properties $H_n \cup \mathds{H}^{(k)}(n,p) \to (\tilde K^{(k)}_s,\tilde K^{(k)}_t)$ with $t\geq s$. For the sake of a clearer presentation, we restrict ourselves to $3$-uniform hypergraphs and the symmetric case for even~$t$.
We conjecture that Theorem~\ref{thm:main-Ramsey} uncovers the threshold for the Ramsey property in question (see Conjecture~\ref{con:main-Ramsey} below). From here on we restrict ourselves to $3$-uniform hypergraphs and unless stated
otherwise we use the term hypergraph for a $3$-uniform hypergraph.

Our proof of Theorem~\ref{thm:main-Ramsey} relies on two main technical results, which are related to the regularity method for hypergraphs. We present these results in \S\ssign\ref{sec:tuple-lemma}-\ref{sec:variant-hreg} below.

\subsection{A tuple lemma for link graphs}
\label{sec:tuple-lemma} A key feature of the regularity method of graphs is the control over joint neighbourhoods in the regular environment provided by Szemer\'edi's regularity lemma (see, e.g., Lemma~\ref{lem:tuple-graphs} below). For the proof of Theorem~\ref{thm:main-Ramsey}, we establish a similar lemma in the context of the regularity method for hypergraphs.

For a vertex $v$ in a hypergraph $H=(V,E)$, define the \emph{link graph $L_H(v)$} of $v$ to have vertex set $V\sm\{v\}$ and edge set comprised of those pairs of vertices which together with $v$ form a hyperedge in $H$, i.e.,
$E(L_H(v)) = \{uw\colon uvw \in E\}$.
In particular, $e(L_H(v))$ is the vertex degree of $v$ in $H$ and is also denoted by $\deg_H(v)$.
Given a graph $G$ with vertex set $V(G) = V$ we define the \emph{link graph of $v$ supported on $G$} by
\[
	L_H(v,G) = E\big(L_H(v)\big)\cap E(G)\,.
\]

Link graphs are a natural hypergraph extension of vertex neighbourhoods in the context of graphs.
A tuple lemma for hypergraphs would have to control the sizes of the of intersections of link graphs. In that, given a set of vertices $U\subseteq V$, we seek to control the sizes of the \emph{joint link graph} and the \emph{joint link graph supported by $G$} given by
\[
	L_H(U)
	=
	\bigcap_{u \in U} L_H(u)
	\qand
	L_H(U, G)
	=
	\bigcap_{u \in U} L_H(u,G)\,,
\]
respectively.
For a random hypergraph $H=(V,E)$ with edge density $d$, one would expect $|L_H(U)|\sim d^{|U|} \binom{|V|}{2}$ to hold with high probability. Our tuple lemma asserts that in the regular environment for hypergraphs this random intuition can be transferred to the deterministic situation. (We defer the definitions concerning regular hypergraphs to Section~\ref{sec:pre}.)

\begin{proposition}[Tuple lemma for joint links]
	\label{lem:tuple}
	For every $t \geq 2$ and  $\eps$, $d_3 > 0$, there exists a $\delta_3>0$ such that for every $d_2>0$
	there exist $\delta_2 > 0$ and $r\geq 1$ such that the following holds for sufficiently large,  pairwise 
	disjoint sets $X$, $Y$, and $Z$.

	Let $H = (X \discup Y \discup Z, E_H)$ be a tripartite hypergraph which is $(\delta_3,d_3,r)$-regular with respect to a $(\delta_2,d_2)$-triad
	$P= (X \discup Y \discup Z,E_P)$. Then, all but at most $\eps |X|^t$ of the $t$-tuples
	of vertices $X'=\{x_1,\dots,x_t\}\subseteq X$ satisfy
	\begin{equation}\label{eq:tuple}
		\left|L_H(X',P)-d_3^t d_2^{2t+1}|Y||Z|\right|
		\leq
		\eps d_3^td_2^{2t+1}|Y||Z|\,.
	\end{equation}
\end{proposition}
Our proof of Proposition~\ref{lem:tuple} extends to all hypergraph uniformities.
Alternatives to Proposition~\ref{lem:tuple} exerting some control over the sizes of joint link graphs of vertex
tuples whilst relying on weaker versions of the hypergraph regularity do exist. Such alternatives are provided in Section~\ref{sec:concluding} (see Lemma~\ref{lem:weak-tuple-lemma} and Proposition~\ref{prop:DRC}).

\subsection{A variant of the hypergraph regularity lemma}
\label{sec:variant-hreg}
The second main technical lemma is a new variant of the hypergraph regularity lemma established in~\cite{SRL}. The necessary definitions are deferred to Section~\ref{sec:pre}.

\begin{proposition}[Variant of the regularity lemma for hypergraphs]
	\label{lem:ghrl}
	For every $\delta_3>0$ and functions  $\delta_2\colon \mathds{N} \to (0,1]$, $r\colon \mathds{N}^2 \to \mathds{N}$, and constants
	$\ell_0$, $t_0$, and $s \in \mathds{N}$, there exist $n_0$ and $T \in \mathds{N}$ such that for every $n \geq n_0$
	and every family $(H_1,\ldots, H_s)$ of $n$-vertex hypergraphs satisfying $V = V(H_1) = \dots = V(H_s)$,
	there are integers $t$ and $\ell$ satisfying $t\geq t_0$ and $\ell \geq \ell_0$,
	a vertex partition $\mathcal{V}$ with $V_1 \discup \cdots \discup V_t=V$
	and an $\ell$-equitable partition $\mathcal{B}$ with respect to $\mathcal{V}$
	such that the following properties hold.
	\begin{description}
		\item [(R.1)] $|V_1| \leq |V_2| \leq \cdots \leq |V_t| \leq |V_1|+1$,
		\item [(R.2)] for all $1 \leq i < j \leq t$ and $\alpha \in [\ell]
		      $, the bipartite $2$-graph $B^{ij}_\alpha$ is $
			      (\delta_2(\ell),1/\ell)$-regular,
		\item [(R.3)] $H_i$ is $\delta_2(\ell)$-weakly regular
		      with respect to $\mathcal{V}$ for every $i \in
			      [s]$, and
		\item [(R.4)] $H_i$ is $(\delta_3,r(t,\ell))$-regular with respect to $\mathcal{B}$ for every $i \in [s]$.
	\end{description}
\end{proposition}

In Proposition~\ref{lem:ghrl} there is a combination of the environment of the hypergraph regularity lemma~\cite{SRL} (see Lemma~\ref{lem:shrl}) and the so-called weak hypergraph regularity lemma (see Lemma~\ref{lem:weak} below), which is the straightforward extension of Szemer\'edi's regularity for graphs. A lemma of similar spirit can be found in the work of Allen, Parczyk, and Pfenninger~\cite{APV21}.

In the sequel, these hypergraph regularity lemmata are distinguished by referring to these as the \emph{Strong Lemma} and \emph{Weak Lemma}, respectively. The difference between the Strong Lemma and Proposition~\ref{lem:ghrl} is Property~(R.3). The former, when applied to dense hypergraphs, provides access to
triads $P$ set over a vertex set, say, $X \discup Y \discup Z$ with respect to which the regularised hypergraphs is $(\delta_3,d,r)$-regular. This, in turn, provides $\zeta$-weak regularity control for $\zeta =\delta_3^{1/3}$, by which we mean the ability to control the hyperedge distribution of the hypergraphs along sets $X' \subseteq X$, $Y' \subseteq Y$, and $Z' \subseteq Z$ satisfying $|X'| \geq \zeta|X|$, $|Y'| \geq \zeta|Y|$, and $|Z'| \geq \zeta |Z|$.

The added Property~(R.3), however, provides weak regularity control over vertex sets with much smaller density. In fact,
there the control $\delta_2$ is allowed to be a function of $\ell$ and the quantification of the Strong Lemma leads
to $\delta_3\gg\ell^{-1}$.

\subsection*{Organisation} In Section~\ref{sec:pre}, we collect definitions, notation, and results pertaining to the regularity method. In Section~\ref{sec:mono}, we prove Theorem~\ref{thm:main-Ramsey}. In Section~\ref{sec:tuple}, we prove Proposition~\ref{lem:tuple}, the aforementioned tuple property for the Strong Lemma. In Section~\ref{sec:ghrl}, we prove our new variant of the Strong Lemma, namely Proposition~\ref{lem:ghrl}.
Concluding remarks concerning alternatives of Proposition~\ref{lem:tuple} and discussions pertaining to future research can be found in Section~\ref{sec:concluding}. Finally, the threshold for the property $\tilde K^{(3)}_t \subseteq H \cup \mathds{H}^{(3)}(n,p)$ in dense hypergraphs~$H$ is established in Appendix~\ref{app:Turan}.

\subsection*{Notational remark} Throughout, we often write the enumeration of a result in the subscripts of the constants that it presides over. For instance, the constant $t_0$ in Proposition~\ref{lem:ghrl} becomes~$t_{\ref{lem:ghrl}}$ and the constant $\delta_3$ in the same lemma is written $\delta_{\ref{lem:ghrl}}^{(3)}$ and so on. This aids in keeping track of the various constants encountered throughout the proofs.

\section{Preliminaries}\label{sec:pre}
Let $V$ be a finite set. A partition $\mathcal{U}$ of $V$ given by $V = U_1 \discup \cdots \discup U_r$ is said to be {\em equitable} if $|U_1| \leq |U_2| \leq \cdots \leq |U_r| \leq |U_1|+1$. Given an additional partition of $V$, namely $\mathcal{V}$, of the form $V = V_1 \discup \cdots \discup V_\ell$, we say that $\mathcal{V}$ {\em refines} $\mathcal{U}$, and write $\mathcal{V} \prec \mathcal{U}$, if for every $i \in [\ell]$ there exists some $j \in [r]$ such that $V_i \subseteq U_j$ holds. For $k \geq 2$, write $K^{(k)}(\mathcal{U})$ to denote the complete $|\mathcal{U}|$-partite $k$-uniform hypergraph whose vertex set is $V$ and whose edge set is given by all sets of $V^{(k)}=\{K\subseteq V\colon |K|=k\}$ meeting every member of $\mathcal{U}$ (termed \emph{cluster} hereafter) in at most one vertex. If $\mathcal{U} = \{U,U'\}$ consists of only two clusters, then we abbreviate $K^{(2)}(\mathcal{U})$ to~$K^{(2)}(U,U')$. We write $K^{(2)}(V)$ to denote the complete graph whose vertex  set is~$V$.

\subsection{Graph regularity}

Let $d,\delta >0$ be given. A bipartite $2$-graph $G = (X \discup Y,E)$ is said to be $(\delta,d)$-{\em regular} if
$$
	e_G(X',Y')  = d|X'||Y'| \pm \delta |X||Y|
$$
holds\footnote{Given $x,y,z \in \mathds{R}$, we write $x = y \pm z$ if $y-z \leq x \leq y+z$.} for every $X' \subseteq X$ and $Y' \subseteq Y$. If $d$ coincides with the edge density of $G$, i.e. $d=\frac{e(G)}{|X||Y|}$, then we abbreviate $(\delta,d)$-regular to $\delta$-regular. It follows directly from the definition that $G$ is a $(\delta,d)$-regular bipartite graph if, 
and only if, its  (bipartite) complement is $(\delta,1-d)$-regular. 

A tripartite $2$-graph $P$ with vertex  set $V(P) = X \discup Y \discup Z$ is said to be a $(\delta,d)$-{\em triad}, if $P[X,Y]$, $P[Y,Z]$, and $P[X,Z]$ are all $(\delta,d)$-regular. For a $2$-graph $G$, let $\mathcal{K}_3(G)$ denote the family of members of $V(G)^{(3)}$ spanning a triangle in $G$. We shall employ the well known 
triangle countling lemma (see, e.g., \cite{FR02}*{Fact~A}).

\begin{lemma}[Triangle counting lemma]\label{lem:tri-cnt}
	Let $d > 0$, let $0 < \delta < d/2$, and let $P$ be a $(\delta,d)$-triad with vertex  set $V(P) = X \discup Y \discup Z$. Then,
	$$
		(1-2\delta)(d-\delta)^3 |X||Y||Z| \leq |\mathcal{K}_3(P)| \leq ((d+\delta)^3 + 2\delta) |X||Y||Z|.
	$$
	In particular, if $d \leq 1/2$, then
	\begin{equation}\label{eq:tri-cnt}
		|\mathcal{K}_3(P)| = (d^3\pm 4\delta)|X||Y||Z|
	\end{equation}
	holds.\qed
\end{lemma}

We shall also use the variant of the triangle counting lemma with only two of the bipartite graphs being regular and its proof is included for completeness.

\begin{lemma}\label{lem:two-sided-tri-cnt}
	Let $P = (X \discup Y \discup Z,E_P)$ be a tripartite $2$-graph such that $P[X,Y]$ and $P[X,Z]$ are both $(\delta,d)$-regular. In addition, let $X' \subseteq X$ be a set of size $|X'| \geq \delta |X|$. Then,
	$$
		(d-\delta)d|X'| e(P[Y,Z]) - 2\delta|X||Y||Z| \leq |\mathcal{K}_3(P,X')| \leq (d+\delta)d|X'| e(P[Y,Z]) + 2\delta|X||Y||Z|
	$$
	holds, where $\mathcal{K}_3(P,X')$ denotes the set of triangles of $P$ meeting $X'$.
\end{lemma}
\begin{proof}
	Let $Y' \subseteq Y$ consist of all vertices $y \in Y$ satisfying $\deg_P(y,X') \geq (d-\delta)|X'|$; note that $|Y'| \geq (1 - \delta) |Y|$ holds by Lemma~\ref{lem:tuple-graphs}. We may then write
	\begin{align*}
		|\mathcal{K}_3(P,X')| & \geq  \sum_{y \in Y'}  \Big(d(d-\delta)|X'| \deg_P(y,Z) - \delta|X||Z|\Big)                                                  \\
		                      & = d(d-\delta)|X'| \left(\sum_{y \in Y} \deg_P(y,Z) - \sum_{y \in Y \sm Y'} \deg_P(y,Z)\right) - \sum_{y \in Y'}\delta |X||Z| \\
		                      & \geq d(d-\delta)|X'|e(P[Y,Z]) - d(d-\delta)\delta|X||Y||Z| - \delta |X||Y||Z|                                                \\
		                      & \geq d(d-\delta)|X'|e(P[Y,Z]) - 2\delta|X||Y||Z|.
	\end{align*}

	Next, we prove the upper bound. Let $Y'' \subseteq Y$ consist of all vertices $y \in Y$ satisfying $\deg_P(y,X') \leq (d+\delta)|X'|$; note that $|Y''| \geq (1 - \delta) |Y|$ holds by Lemma~\ref{lem:tuple-graphs}. We may then write
	\begin{align*}\pushQED{\qed}
		|\mathcal{K}_3(P,X')| & \leq \sum_{y \in Y''} \Big(d(d+\delta)|X'|\deg_P(y,Z) + \delta |X||Z|\Big) + \sum_{y \in Y \sm Y''} |X'||Z| \\
		                      & \leq d(d+\delta)|X'|\left(\sum_{y \in Y}\deg_P(y,Z) - \sum_{y \in Y \sm Y''} \deg_P(y,Z) \right)            \\
		                      & \hspace{5cm}+ \sum_{y \in Y''}\delta |X||Z| + \sum_{y \in Y \sm Y''} |X||Z|                                 \\
		                      & \leq d(d+\delta)|X'| e(P[Y,Z]) + 2\delta |X||Y||Z|\,.\qedhere
	\end{align*}
\end{proof}

The next lemma is commonly referred to as the {\em Slicing Lemma} (see, e.g.,~\cite{KS96}*{Fact~1.5}).

\begin{lemma}[Slicing lemma]\label{lem:slicing} 
	Let $d = d_{\ref{lem:slicing}}$, let $\delta= \delta_{\ref{lem:slicing}} > 0$, and let $G=(A \discup B,E)$ be a $(\delta,d)$-regular bipartite graph. Let $\delta \leq \alpha = \alpha_{\ref{lem:slicing}} \leq 1$, and let $A' \subseteq A$ and $B' \subseteq B$ be sets of sizes $|A'| \geq \alpha |A|$ and $|B'| \geq \alpha |B|$. Then, $G[A',B']$ is $(\delta',d')$-regular where $\delta' = \max\{\delta/\alpha,2\delta\}$ and $d' = d \pm \delta$.\qed
\end{lemma}

The {\it tuple property} of dense regular bipartite graphs, also referred to as the {\it intersection property}, reads as follows
(see~\cite{KS96}*{Fact~1.4}).

\begin{lemma}[Tuple lemma for graphs]\label{lem:tuple-graphs}
	Let $G = (X \discup Y,E)$ be a $\delta$-regular bipartite graph of edge density $d > 0$. Then, all but at most
	$2\delta\ell|X|^\ell$ of the tuples $\{x_1,\ldots,x_\ell\}\subseteq  X$ satisfy
	\begin{equation}\label{eq:tuple-graphs}
		|N_G(x_1,\ldots,x_\ell,Y')| = \left|\{y \in Y'\colon x_i y \in E(G)\; \text{for all}\; i \in [\ell]\}\right| = (d\pm\delta)^\ell |Y'|,
	\end{equation}
	whenever $Y' \subseteq Y$ satisfies $(d-\delta)^{\ell-1}|Y'| \geq \delta |Y|$.\qed
\end{lemma}

\subsection{Hypergraph regularity}\label{sec:HRL}
A direct generalisation of the notion of $\delta$-regularity, defined in the previous section for $2$-graphs, reads as follows. Let $d,\delta > 0$. A tripartite hypergraph $H = (X \discup Y \discup Z,E)$ is said to be $(\delta,d)$-{\em weakly regular} if
$$
	e_H(X',Y',Z') = d|X'||Y'||Z'| \pm \delta |X||Y||Z|
$$
holds whenever $X' \subseteq X$, $Y' \subseteq Y$, and $Z' \subseteq Z$. If $d = \frac{e(H)}{|X||Y||Z|}$, then we abbreviate $(\delta, d)$-weakly regular to $\delta$-weakly regular.

Given a partition $\mathcal{V}$ of a finite set $V$ defined by $V = V_1\discup \cdots \discup V_t$, a hypergraph $H$ with $V(H) = V$ is said to be $\delta$-{\em weakly regular with respect to $\mathcal{V}$} if $H[X,Y,Z]$\footnote{$H[X,Y,Z]$ is the subgraph of $H$ over $X \discup Y \discup Y$ whose edge set is $\{\{x,y,z\} \in E(H)\colon x \in X, y \in Y, z \in Z\}$.} is $\delta$-weakly regular with respect to all but at most $\delta \binom{t}{3}$ triples  $\{X,Y,Z\} \in \mathcal{V}^{(3)}$.
We state the straightforward adaptation of Szemer\'edi's graph regularity lemma~\cites{KSSS, KS96, Sz78}.

\begin{lemma}[Weak hypergraph regularity lemma]\label{lem:weak}
	For every $\delta = \delta_{\ref{lem:weak}} > 0$ and positive integers $s = s_{\ref{lem:weak}}$, $t = t_{\ref{lem:weak}}$, and $h = h_{\ref{lem:weak}}$ satisfying $t \geq h$, there exist positive integers $n_0$ and $T = T_{\ref{lem:weak}}$ such that the following holds whenever $n \geq n_0$. Let $(H_1,\ldots, H_s)$ be a sequence of $n$-vertex hypergraphs, all on the same vertex  set, namely $V$, and let $\mathcal{U}= \mathcal{U}_{\ref{lem:weak}}$ be a vertex  partition of $V$ given by $V = U_1 \discup \ldots \discup U_h$. Then, there exists an equitable vertex partition~$\mathcal{V}$, given by $V = V_1 \discup V_2 \discup \cdots \discup V_{t'}$, where $t \leq t' \leq T$, such that $\mathcal{V} \prec \mathcal{U}$ and, moreover, $H_i$ is $\delta$-weakly regular with respect to $\mathcal{V}$ for every $i \in [s]$.\qed
\end{lemma}

We proceed to the statement of the Strong hypergraph Regularity Lemma for hypergraphs following the formulation seen in~\cite{SRL}. Given a $2$-graph $G$, the {\em relative density} of a hypergraph~$H$ with vertex  set $V(H) = V(G)$, with respect to $G$ is given by
\begin{align} \label{eq::relativeDensity}
	d(H|G) = \frac{|E(H) \cap \mathcal{K}_3(G)|}{|\mathcal{K}_3(G)|}.
\end{align}
For $\delta, d > 0$ and a positive integer $r$, a tripartite hypergraph $H = (X \discup Y \discup Z,E_H) $ is said to be $(\delta,d,r)$-{\em regular} with respect to a tripartite $2$-graph $P = (X \discup Y \discup Z,E_P)$ if
\begin{equation}\label{eq:str-reg}
	\bigg| \Big|\bigcup_{i=1}^r (E_H \cap \mathcal{K}_3(Q_i)) \Big| - d \Big|\bigcup_{i=1}^r \mathcal{K}_3(Q_i) \Big|  \bigg| 
	\leq 
	\delta \big|\mathcal{K}_3(P)\big|
\end{equation}
holds for every family of, not necessarily disjoint, subgraphs $Q_1,\ldots,Q_r \subseteq P$ satisfying
$$
	\bigg|\bigcup_{i = 1}^r \mathcal{K}_3(Q_i) \bigg| 
	\geq 
	\delta \big|\mathcal{K}_3(P)\big| > 0\,.
$$

Let $V$ be a finite set and let $\mathcal{V}$ be a partition $V_1 \discup \ldots \discup V_h$ of $V$, where $h$ is some positive integer. Given an integer $\ell \geq 1$, a partition $\mathcal{B}$ of $K^{(2)}(\mathcal{V})$ is said to be $\ell$-{\em equitable with respect to $\mathcal{V}$} if it satisfies the following conditions:
\begin{description}
	\item [(B.1)] every $B \in \mathcal{B}$ satisfies $B
		      \subseteq K^{(2)}(V_i,V_j)$ for some distinct $i,j \in
		      [h]$; and
	\item [(B.2)] for any distinct $i,j \in [h]$,
	      precisely $\ell$ members of $\mathcal{B}$ partition
	      $K^{(2)}(V_i,V_j)$.
\end{description}
We view partitions of $K^{(2)}(\mathcal{V})$ as partitions of $V^{(2)}$ under the {\it agreement}\footnote{We appeal to this agreement in the formulations of Lemmata~\ref{lem:index-inc} and~\ref{lem:approx}.} that the set $\{K^{(2)}(V_i)\colon i \in [h]\}$ of complete graphs is added to the former; such an addition of cliques does not hinder the equitability notion defined in~(B.2); it does violate (B.1), but this will not harm our arguments. Moreover, it is under this agreement that we say that a partition of~$V^{(2)}$ refines a partition of $K^{(2)}(\mathcal{V})$.

For distinct indices $i,j \in [h]$, the partition of $K^{(2)}(V_i,V_j)$ induced by $\mathcal{B}$ is denoted by 
\[
	\mathcal{B}^{ij} 
	= 
	\{B^{ij}_\alpha = (V_i \discup V_j,E^{ij}_\alpha)\colon \alpha \in [\ell]\}\,.
\] 
The {\em triads} of $\mathcal{B}$ are the tripartite $2$-graphs having the form
$$
	B^{ijk}_{\alpha\beta\gamma}= (V_i \discup V_j \discup V_k, E^{ij}_\alpha \discup E^{ik}_\beta \discup E^{jk}_\gamma),
$$
where $i,j,k \in [h]$ are distinct and $\alpha, \beta, \gamma \in [\ell]$. Recall that a triad is called a $(\delta,d)$-{\em triad} if each of the three bipartite graphs comprising it is $(\delta,d)$-regular. A hypergraph $H$ with vertex  set $V(H) = V$ is said to be $(\delta,r)$-{\em regular} with respect to $\mathcal{B}$ if
\[
	\Bigg|\Bigg\{
	\bigcup_{\substack{1\leq i < j < k\leq h \\
			\alpha,\beta, \gamma \in [\ell]}}
	\mathcal{K}_3(B^{ijk}_{\alpha \beta \gamma})\colon H_{ijk} \ \text{is not}\
	\big(\delta, d(H|B^{ijk}_{\alpha \beta \gamma}),r\big)\text{-regular w.r.t.} \
	B^{ijk}_{\alpha \beta \gamma}
	\Bigg\}\Bigg|
	\leq
	\delta |V|^3\,,
\]
where $H_{ijk} = H[V_i \discup V_j \discup V_k]$. A formulation of the Strong Lemma~\cite{SRL}*{Theorem~17} for hypergraphs, reads as follows.
\begin{lemma}[Strong hypergraph regularity lemma]\label{lem:shrl}
	For all $0< \delta_3 \in \mathds{R}$, $\delta_2\colon \mathds{N} \to (0,1]$, $r\colon \mathds{N}^2 \to \mathds{N}$, and $s,t,\ell \in \mathds{N}$, there exist $n_0$, $T \in \mathds{N}$ such that for every $n \geq n_0$ and every sequence of $n$-vertex hypergraphs $(H_1,\ldots,H_s)$, satisfying $V = V(H_1) = \cdots = V(H_s)$, there are $t', \ell' \in \mathds{N}$ satisfying $t \leq t' \leq T$ and $\ell \leq \ell' \leq T$, a vertex partition $V = V_1 \discup \cdots \discup V_{t'}$, namely~$\mathcal{V}$, and an $\ell'$-equitable partition $\mathcal{B}$ with respect to $\mathcal{V}$  such that the following properties hold.
	\begin{description}
		\item [(S.1)] $|V_1| \leq |V_2| \leq \cdots \leq |V_{t'}| \leq |V_1|+1$;

		\item [(S.2)] for all $1 \leq i < j \leq t'$ and $\alpha \in [\ell']
		      $, the bipartite $2$-graph $B^{ij}_\alpha$ is $(\delta_2(\ell'),1/\ell')$-regular; and

		\item [(S.3)] $H_i$ is $(\delta_3,r(t',\ell'))$-regular with respect to $\mathcal{B}$ for every $i \in [s]$.\qed
	\end{description}
\end{lemma}

\section{Monochromatic expanded cliques}\label{sec:mono}
In this section, we prove Theorem~\ref{thm:main-Ramsey}.
The required Ramsey properties of $\mathds{H}^{(3)}(n,p)$ are collected in Section~\ref{sec:rnd-props} and a proof of Theorem~\ref{thm:main-Ramsey} can be found in Section~\ref{sec:proof-main-Ramsey}. For an integer $t \geq 3$, the $t$ vertices of $\tilde K^{(k)}_t$ having their $1$-degree strictly larger than one are called the {\em branch vertices} of $\tilde K^{(k)}_t$. Set
$$
	v(t)= v(\tilde K^{(3)}_t)=t+\binom{t}{2}\quad \text{and}\quad  e(t)= e(\tilde K^{(3)}_t)=\binom{t}{2}\,.
$$

\subsection{Properties of random hypergraphs}\label{sec:rnd-props}
The main goal of this section is to establish Proposition~\ref{prop:rnd-prop} which is an adaptation of~\cite{DT19}*{Theorem~2.10}. This proposition collects the Ramsey properties of $\mathds{H}^{(3)}(n,p)$ that will be utilised throughout our proof of Theorem~\ref{thm:main-Ramsey}.

A $k$-uniform hypergraph $H$ is said to be {\em balanced}
if $m_k(H) = d_k(H)$ holds; if all proper subgraphs $F$ of~$H$ satisfy $m_k(F) < m_k(H)$, then $H$ is said to be {\em strictly balanced}. It is not hard to verify that expanded cliques are strictly balanced. In particular,
$$
	m_k\big(\tilde K^{(k)}_t\big) = \frac{\binom{t}{2}-1}{t+(k-2)\binom{t}{2}-k}
$$
holds for any $k \geq 2$ and $t \geq 3$. In the special case $k=3$ we obtain
\begin{equation}\label{eq:Mt}
	m_3\big(\tilde K^{(k)}_t\big)= \frac{t^2-t-2}{t^2+t-6} = 1 - \frac{2t-4}{t^2+t-6} < 1\,,
\end{equation}
that is, $3$-uniformly expanded cliques are \emph{sparse}. Note that this is in contrast to graph cliques (on at least 3 vertices) whose 2-density is larger than one.
For a simpler notation we define for any integer $t \geq 2$
\[
	m(t)= m(\tilde K^{(3)}_t)
	\qand
	M_t = m_3(\tilde K^{(3)}_t)\,.
\]
Similarly for integers $t_1$, $t_2 \geq 2$ we set
\[
	M_{t_1, t_2}
	=
	M_{t_2, t_1}
	=
	m_3 \big(\tilde K^{(3)}_{t_1}, \tilde K^{(3)}_{t_2} \big)\,.
\]

Let $H_1$ and $H_2$ be two $k$-uniform hypergraphs, each with at least one edge and such that $m_k(H_1) \geq m_k(H_2)$.
If $m_k(H_1) = m_k(H_2)$, then $m_k(H_1,H_2) = m_k(H_1)$ and otherwise 
\[
	m_k(H_2) < m_k(H_1,H_2) < m_k(H_1)
\] 
holds. The $k$-uniform hypergraph $H_1$ is said to be {\em strictly balanced with respect to $m_k(\cdot,H_2)$} if no proper subgraph $F \subsetneq H_1$ maximises~\eqref{eq:asym-density}. For instance, it is not hard to verify  that $\tilde K^{(3)}_t$ is strictly balanced with respect to $m_3(\cdot, \tilde K^{(3)}_{t/2})$, assuming $t \geq 4$  is even.

Let $F$ and $F'$ be $k$-uniform hypergraphs and let $\mu = \mu(n)$ be given. An $n$-vertex $k$-uniform hypergraph $H$ is said to be $(F,\mu)$-{\em Ramsey} if $H[U] \to (F)$ holds for every $U \subseteq V(H)$ is of size $|U| \geq \mu n$. Similarly, $H$ is said to be $(F, F', \mu)$-{\em Ramsey} if $H[U] \to (F,F')$ holds for every $U \subseteq V(H)$ of size $|U| \geq \mu n$. Given $\mathcal{F} \subseteq V(H)^{(v(F))}$ and $\mathcal{F}' \subseteq V(H)^{(v(F'))}$, we say that $H$ is $(F,F')$-{\em Ramsey with respect to $(\mathcal{F},\mathcal{F}')$} if any $2$-colouring of $E(H)$ yields a monochromatic copy $K$ of $F$ (in the first colour) with $V(K) \notin \mathcal{F}$ or a monochromatic copy $K'$ of $F'$ (in the second colour) with $V(K') \notin \mathcal{F}'$.

\begin{proposition}\label{prop:rnd-prop}
	For every even integer $t \geq 4$ a.a.s.\  the binomial random hypergraph $R\sim\mathds{H}^{(3)}(n,p)$ 
	satisfies the following properties.
	\begin{enumerate}
		\item [{\em (P.1)}] There are constants $
			      \gamma_{\ref{prop:rnd-prop}}=\gamma_{\ref{prop:rnd-prop}}
			      (t)$ and $C^{(1)}_{\ref{prop:rnd-prop}}=C^{(1)}
			      _{\ref{prop:rnd-prop}}(t)$ such that if $\mathcal{F}_1
			      \subseteq V(R)^{(v(t))}$ and $\mathcal{F}_2 \subseteq V(H)^{(v(t/2))}$
			       satisfy
		      $|\mathcal{F}_1| \leq \gamma_{\ref{prop:rnd-prop}}
			      n^{v(t)}$ and $|\mathcal{F}_2| \leq
			      \gamma_{\ref{prop:rnd-prop}} n^{v(t/2)}$,
		      then $R$ is $(\tilde K^{(3)}_t,\tilde K^{(3)}_{t/2})$-Ramsey with respect to
		      $(\mathcal{F}_1,\mathcal{F}_2)$,
		      whenever $p= p(n) \geq C^{(1)}_{\ref{prop:rnd-prop}}
			      n^{-1/M_{t,t/2}}$.
		\item [{\em (P.2)}] For every fixed $\mu > 0$, there exists
		      a constant $C^{(2)}_{\ref{prop:rnd-prop}} =
			      C^{(2)}_{\ref{prop:rnd-prop}}(\mu,t)$ such that $R$ is
		      $(\tilde K^{(3)}_{t-1},\mu)$-Ramsey, whenever $p= p(n)
			      \geq C^{(2)}_{\ref{prop:rnd-prop}} n^{-1/M_{t-1}}$.
		\item [{\em (P.3)}] For every fixed $\mu > 0$, there exists a
		      constant $C^{(3)}_{\ref{prop:rnd-prop}} =
			      C^{(3)}_{\ref{prop:rnd-prop}}(\mu,t)$ such that $R$ is
		      $(\tilde K^{(3)}_t,\tilde K^{(3)}_{t/2},\mu)$-Ramsey, whenever
		      $p= p(n)
			      \geq C^{(3)}_{\ref{prop:rnd-prop}} n^{-1/M_{t,t/2}}$.
	\end{enumerate}
\end{proposition}
\begin{proof}
A straightforward albeit somewhat tedious calculation shows that $M_{t,t/2} \geq M_{t-1}$ holds for every even integer $t \geq 4$. It thus follows that Properties~(P.1) and~(P.3) are the most stringent in terms of the bound these impose on $p$. Hence, if 
\[
	p = p(n) \geq \max \big\{C^{(1)}_{\ref{prop:rnd-prop}}, C^{(3)}_{\ref{prop:rnd-prop}}\big\} \cdot n^{-1/M_{t,t/2}}\,,
\] 
then a.a.s.\ $H$ satisfies Properties~(P.1), (P.2), and~(P.3) simultaneously.
\end{proof}

Property~(P.1) is modelled after~\cite{DT19}*{Theorem~2.10(i)}; Properties~(P.2) and~(P.3) are both specific instantiations of~\cite{DT19}*{Theorem~2.10(ii)}.
The aforementioned results of~\cite{DT19} handle $2$-graphs only. Nevertheless, proofs of Properties (P.1-3) can be attained by straightforwardly adjusting the proofs of their aforementioned counterparts in~\cite{DT19}*{Theorem~2.10} so as to accommodate the transition from $2$-graphs to hypergraphs. Theorem~2.10 in~\cite{DT19} requires that the maximal $2$-densities of the two (fixed) configurations would both be at least one; this can be omitted in our setting. Indeed, this condition is imposed
in~\cite{DT19}*{Theorem~2.10} in order to handle setting~(a) in that theorem where the maximal $2$-densities of the two configurations {\it coincide}; by~\eqref{eq:Mt}, this is not an issue in our case. The fact that $\tilde K^{(3)}_t$ is strictly balanced with respect to $m_3(\cdot,\tilde K^{(3)}_{t/2})$ is required by setting~(b) appearing in~\cite{DT19}*{Theorem~2.10}.

\subsection{Proof of Theorem~\ref{thm:main-Ramsey}}\label{sec:proof-main-Ramsey}

We commence our proof of Theorem~\ref{thm:main-Ramsey} with a few observations facilitating our arguments; proofs of these observations are included for completeness.

\begin{fact} \label{cor:support}
	Let $d \in (0,1]$, let $G = (A \discup B,E)$ be a bipartite graph with $e(G) \geq d|A||B|$, and let $k \leq d |B|/2$ be a positive integer. Then, $\left|\{v \in A\colon \deg_G(v) \geq k\} \right| \geq d|A|/2$.
\end{fact}
\begin{proof}
	Let $A_k = \{v \in A \colon \deg_G(v) \geq k\}$ and suppose for a contradiction that $|A_k| < d |A|/2$. Then,
	$$
		e(G) < k |A| + |A_k| |B| < d |A| |B|/2 + d |A| |B|/2 \leq e(G)
	$$
	which is clearly a contradiction.
\end{proof}

The next lemma captures the phenomenon of {\it supersaturation} (first\footnote{Rademacher (1941, unpublished) was first to prove that every $n$-vertex graph with $\lfloor n^2/4\rfloor +1$ edges contains at least $\lfloor n/2 \rfloor$ triangles } recorded in~\cites{Erdos55,Erdos62a,Erdos62b}) for bipartite graphs; to facilitate future references, we phrase this lemma with the host graph being bipartite as well.
\begin{lemma}\label{lem:supersaturation}
	For every bipartite graph $K$ and every $d \in (0,1)$, there exists a constant $\zeta = \zeta_{\ref{lem:supersaturation}} > 0 $ and a positive integer $n_0$ such that every $n$-vertex bipartite graph $G = (A \discup B,E)$ satisfying $n \geq n_0$, $|A| \leq |B| \leq |A|+1$, and $e(G) \geq d|A||B|$ contains at least $\zeta n^{v(K)}$ distinct copies of $K$.\qed
\end{lemma}

\begin{fact}\label{obs:disjoint-copies}
	For every graph $K$ and every $d \in (0,1)$, there exists a constant $\xi = \xi_{\ref{obs:disjoint-copies}} >0$ and an integer $n_0$ such that the following holds whenever $n \geq n_0.$ If an $n$-vertex graph $G$ contains $d n^{v(K)}$ distinct copies of $K$, then it contains at least $\xi n$ pairwise vertex  disjoint copies of $K$.
\end{fact}

\begin{proof}
	Any given copy of $K$ meets $O \left(n^{v(K)-1} \right)$ copies of $K$.
\end{proof}

After these preparations we can now present the proof of the main result.

\begin{proof}[Proof of Theorem~\ref{thm:main-Ramsey}]
	Given $d$, $t$, and for sufficiently large $n$ let $H=H_n$ be given as in the premise of Theorem~\ref{thm:main-Ramsey}. We set
	\begin{equation}\label{eq:d3-eps}
		0 < d_3 \ll d \qand 0 < \eps \ll \min \left\{d_3^{v(t/2)}, \gamma_{\ref{prop:rnd-prop}}(t) \right\}.
	\end{equation}
	Proposition~\ref{lem:tuple} applied with $t_{\ref{lem:tuple}} = v(t/2)$, $\eps_{\ref{lem:tuple}} = \eps$, and $d^{(3)}_{\ref{lem:tuple}} = d_3$, yields the existence of a constant
	\begin{equation}\label{eq:delta3-again}
		0 < \delta_3 = \delta^{(3)}_{\ref{lem:tuple}}(v(t/2),\eps,d_3) \ll d_3
	\end{equation}
	as well as the functions
	$$
		\delta'_2(x)=\delta^{(2)}_{\ref{lem:tuple}}(x,t_{\ref{lem:tuple}},\eps,d_3,\delta_3) \; \text{and} \;
		r(x) = r_{\ref{lem:tuple}}(x,t_{\ref{lem:tuple}}, \eps,d_3, \delta_3),
	$$
	where $\delta'_2 \colon \mathds{R} \to (0,1]$ and $r \colon \mathds{N} \to \mathds{N}$. Define $\delta_2\colon \mathds{N} \to (0,1]$ such that
	\begin{equation}\label{eq:delta2-r-func}
		0< \delta_2(x) \ll \min\left\{\delta'_2(x), \frac{d_3^{2v(t/2)}}{v(t/2) \cdot x^{6\cdot(2v(t/2)+1)}}\right\} \end{equation}
	holds for every $x \in \mathds{N}$. Proposition~\ref{lem:ghrl}, applied with
	\begin{equation}\label{eq:ghrl-consts}
		H_1 = \cdots = H_s = H,\; \delta^{(3)}_{\ref{lem:ghrl}} = \delta_3,\; \delta^{(2)}_{\ref{lem:ghrl}}= \delta_2, \;  r_{\ref{lem:ghrl}} = r\footnote{Formally, $r$ is a function of one integer whereas $r_{\ref{lem:ghrl}}$ is a function of two. However, this ``loss of information'' is a technicality that will not hinder our proof.}, \; \ell_{\ref{lem:ghrl}} \gg d_3^{-1},\; \text{and} \; t_{\ref{lem:ghrl}} \gg d^{-1},
	\end{equation}
	yields the existence of constants $T_{\ref{lem:ghrl}}, t',\ell \in \mathds{N}$ satisfying $t_{\ref{lem:ghrl}} \leq t' \leq T_{\ref{lem:ghrl}}$ and $\ell_{\ref{lem:ghrl}}\leq \ell \leq T_{\ref{lem:ghrl}}$, along with partitions $\mathcal{V} = V_1 \discup \cdots \discup V_{t'} = V(H)$ and ($\mathcal{P}^{ij})_{1 \leq i < j \leq t'}$ satisfying Properties (R.1-4). Set auxiliary constants
	\begin{equation}\label{eq:aux}
		d_2 = 1/\ell \quad \text{and} \quad \eta= \frac{d_3^{v(t/2)}d_2^{2v(t/2)+1}}{2}
	\end{equation}
	and fix
	\begin{equation}\label{eq:mu}
		0 < \mu \ll \frac{\xi_{\ref{obs:disjoint-copies}}(\zeta_{\ref{lem:supersaturation}}(\eta/2))\cdot d_3^{3+2v(t/2)} \cdot d_2^{10+4v(t/2)}}{v(t/2)^2 \cdot T_{\ref{lem:ghrl}}}.
	\end{equation}

	We claim that there exist three distinct clusters $X,Y,Z \in \mathcal{V}$ along with a $(\delta_2(\ell),d_2)$-triad $P= P^{ijk}_{\alpha \beta \gamma}$, with $i,j,k,\alpha,\beta,\gamma$ appropriately defined, satisfying $V(P) = X \discup Y \discup Z$ such that
	$H[X \discup Y \discup Z]$ is $\delta_2(\ell)$-weakly regular and, moreover, $H[X \discup Y \discup Z]$ is $(\delta_3,d_3,r)$-regular with respect to $P$. To see this, note first that at most $t' \binom{\lceil n/t' \rceil}{3} \leq \frac{n^3}{(t')^2} \ll dn^3$ edges of $H$ reside within the members of $\mathcal{V}$, where the last inequality relies on $t' \geq t_{\ref{lem:ghrl}} \gg d^{-1}$, supported by~\eqref{eq:ghrl-consts}.
	Second, by Property~(R.3), the number of edges of $H$ captured within
	$\delta_2(\ell)$-weakly irregular triples $(V_i,V_j,V_k)$, where $i,j,k \in [t']$, is at most $\delta_2(\ell) \cdot (t')^3 \cdot \left(\frac{n}{t'} + 1\right)^3 \leq 2 \delta_2(\ell) n^3 \ll d n^3$, where the last inequality holds by~\eqref{eq:d3-eps} and~\eqref{eq:delta2-r-func}. Third, by Property~(R.4), the number of edges of~$H$ residing\footnote{\label{foot:support} Supported by triangles of such triads.} in
	$(\delta_3,d(H|P^{ijk}_{\alpha\beta\gamma}),r)$-irregular triads $P^{ijk}_{\alpha\beta\gamma}$ is at most $\delta_3 n^3 \ll dn^3$, where the last inequality holds by~\eqref{eq:d3-eps} and~\eqref{eq:delta3-again}.
	Fourth and lastly, it follows by the Triangle Counting Lemma (Lemma~\ref{lem:tri-cnt}) and by~\eqref{eq::relativeDensity}, that the number of edges of $H$ found in $(\delta_2(\ell), d_2)$-triads~$P^{ijk}_{\alpha\beta\gamma}$, where $i,j,k \in [t']$ and $\alpha,\beta, \gamma \in [\ell]$, satisfying $d(H|P^{ijk}_{\alpha\beta\gamma}) < d_3$ is at most
	$$
		(t')^3 \ell^3 d_3 \left(d_2^3 + 4 \delta_2(\ell) \right) \left(\frac{n}{t'} + 1 \right)^3 \leq 2 d_3 \left(\ell^3 d^3_2 + 4 \ell^3 \delta_2(\ell)\right) n^3 \overset{\eqref{eq:aux}}{=}
		(2 + 8 \ell^3 \delta_2(\ell)) d_3 n^3 \ll dn^3,
	$$
	where the last inequality holds by~\eqref{eq:d3-eps} and~\eqref{eq:delta2-r-func}.

	It follows that at least $dn^3/2$ edges of $H$ are captured in $(\delta_2(\ell),d_2)$-triads with respect to which $H$ is $(\delta_3,d_3,r)$-regular and such that $H$ is $\delta_2(\ell)$-weakly regular with respect to the three members of $\mathcal{V}$ defining the vertex  sets of these triads. The existence of $X,Y,Z \in \mathcal{V}$ and $P$ as defined above is then established.
	Throughout the remainder of the proof, we identify $H$ with $H[X \discup Y \discup Z]$.

	Let $\mathcal{F} \subseteq X^{(v(t/2))}$ be the family of all sets $\{x_1,\ldots,x_{v(t/2)}\} \subseteq X$ satisfying
	\begin{equation}\label{eq:F}
		\bigg|\bigcap_{j \in [v(t/2)]}L_H(x_j,P)\bigg| 
		< 
		\big(d_3^{v(t/2)}-\eps\big) d_2^{2v(t/2)+1}|Y||Z|.
	\end{equation}
	Then,
	$$
		|\mathcal{F}| \leq \eps |X|^{v(t/2)} \overset{\eqref{eq:d3-eps}}{\ll} \gamma_{\ref{prop:rnd-prop}}(t)|X|^{v(t/2)}
	$$
	holds by~\eqref{eq:tuple}. This application of the Tuple Lemma is supported by our choice $\ell_{\ref{lem:ghrl}} \gg d_3^{-1}$, seen in~\eqref{eq:ghrl-consts}, ensuring that $d_2 \ll d_3$ holds and thus fitting the quantification of Proposition~\ref{lem:tuple}.
	With foresight (see (C.1) and (C.2) below), let
	$$
		C = \max\left\{C^{(1)}_{\ref{prop:rnd-prop}}(t), C^{(2)}_{\ref{prop:rnd-prop}}(\mu, t),C^{(3)}_{\ref{prop:rnd-prop}}(\mu, t)\right\} \cdot (t')^{1/M_{t,t/2}}
	$$
	and put
	$$
		p = p(n) = C \max \left\{n^{-1/M_{t,t/2}},n^{-1/M_{t-1}} \right\} = C n^{-1/M_{t,t/2}}\,.
	$$
	Proposition~\ref{prop:rnd-prop} then asserts that
	the following properties are all satisfied simultaneously a.a.s.\ whenever $R \sim \mathds{H}^{(3)}(n,p)$; in the following list of properties, whenever an asymmetric Ramsey property is stated, the first colour is assumed to be red and the second colour is assumed to be blue.
	\begin{description}
		\item [(C.1)] $R[X]$ is $(\tilde K^{(3)}_t,\tilde K^{(3)}_{t/2})$-Ramsey
		      with respect to $(\emptyset,\mathcal{F})$;

		\item [(C.2)] $R[X]$ is $(\tilde K^{(3)}_{t/2},\tilde K^{(3)}_t)$-Ramsey
		      with respect to $(\mathcal{F},\emptyset)$;

		\item [(C.3)] $R$ is $(\tilde K^{(3)}_{t-1},\mu)$-Ramsey;

		\item [(C.4)] $R$ is $(\tilde K^{(3)}_t,\tilde K^{(3)}_{t/2},\mu)$-Ramsey;

		\item [(C.5)] $R$ is $(\tilde K^{(3)}_{t/2},\tilde K^{(3)}_t,\mu)$-Ramsey.
	\end{description}
	Fix $R \sim \mathds{H}^{(3)}(n,p)$ satisfying Properties~(C.1-5) and set $\Gamma = H \cup R$.

	Let $\psi$ be a red/blue colouring of $E(\Gamma)$ and suppose for a contradiction that $\psi$ does not yield any monochromatic copy of $\tilde K^{(3)}_t$. For every $v \in V(H)$, let $L^{(r)}_H(v)$ denote the {\em red link graph of $v$ in $H$ under $\psi$}, that is, $L^{(r)}_H(v)$ is a spanning subgraph of $L_H(v)$ consisting of the edges of $L_H(v)$ that together with $v$ yield a red edge of $H$ under $\psi$. Similarly, let $L^{(b)}_H(v)$ denote the {\it blue link graph of $v$ in $H$ under $\psi$}. Note that, for any fixed vertex $v$, these two link subgraphs are edge disjoint.

	We say that blue (respectively, red) is \emph{a majority colour} of $\psi$ in $H$ if 
	\[
		|\{e \in E(H) \colon \psi(e) \textrm{ is blue}\}| 
		\geq 
		|\{e \in E(H) \colon \psi(e) \textrm{ is red}\}|
	\] 
	(respectively, $|\{e \in E(H) \colon \psi(e) \textrm{ is red}\}| \geq |\{e \in E(H) \colon \psi(e) \textrm{ is blue}\}|$).
	\begin{claim}\label{clm:blue-maj}
		If blue is a majority colour of $\psi$ in $H$, then $e\big(L^{(r)}_H(v)\big) \leq \frac{\eta}{2v(t/2)} \cdot |Y||Z|$ holds for every $v \in X$.
	\end{claim}
	\begin{proof}
		Suppose for a contradiction that there exists a vertex $v \in X$ which violates the assertion of the claim.
		The Triangle Counting Lemma (Lemma~\ref{lem:tri-cnt}) coupled with the assumption of $H$ being $(\delta_3,d_3,r)$-regular with respect to the $(\delta_2(\ell), d_2)$-triad $P$ (take $Q_1=\cdots = Q_r = P$ in~\eqref{eq:str-reg}) collectively yield
		\begin{align}
			e(H) & \overset{\phantom{\eqref{eq:tri-cnt}}}{\geq} (d_3 -\delta_3)|\mathcal{K}_3(P)|  \nonumber                                    \\
			     & \overset{\eqref{eq:tri-cnt}}{\geq} (d_3-\delta_3)\left(d_2^3 - 4 \delta_2(\ell)\right)|X||Y||Z| \nonumber                    \\
			     & \overset{\phantom{\eqref{eq:tri-cnt}}}{\geq} \left(d_3 d_2^3 - \delta_3 d_2^3- 4d_3 \delta_2(\ell)\right)|X||Y||Z| \nonumber \\
			     & \overset{\phantom{\eqref{eq:tri-cnt}}}{\geq} \frac{d_3d_2^3}{2} |X||Y||Z|, \label{eq:H-density}
		\end{align}
		where the last inequality is owing to $\delta_3 \ll d_3$ and $\delta_2(\ell) \ll d_2^3$ supported by~\eqref{eq:delta3-again} and~\eqref{eq:delta2-r-func}, respectively. Blue being the majority colour implies that at least $\frac{d_3d_2^3}{4} |X||Y||Z|$ of the edges of $H$ are blue and thus there exists a vertex $u \in Z$ satisfying $e\big(L^{(b)}_H(u)\big) \geq \frac{d_3d_2^3}{4}|X||Y|$; note that $L^{(b)}_H(u) \subseteq X \times Y$. Set
		$$
			A_v  = \left\{z \in Z\colon \deg_{L^{(r)}_H(v)}(z) \geq t \right\} \subseteq Z \quad \text{and} \quad
			A_u = \left\{x \in X\colon \deg_{L^{(b)}_H(u)}(x) \geq t \right\} \subseteq X.
		$$
		Then,
		\begin{equation}\label{eq:Au-Av-size}
			|A_v| \geq \frac{\eta}{4 v(t/2)}|Z|\overset{\eqref{eq:delta2-r-func}}{\geq} \delta_2(\ell) |Z| \quad \text{and}\quad |A_u| \geq \frac{d_3d_2^3}{8}|X| \overset{\eqref{eq:delta2-r-func}}{\geq} \delta_2(\ell) |X|
		\end{equation}
		both hold by Fact~\ref{cor:support}. Since $H$ is $\delta_2(\ell)$-weakly regular, it follows that
		\begin{align}
			e_H(A_u,Y,A_v) & \overset{\eqref{eq:H-density}}{\geq} \left(\frac{d_3d_2^3}{2}\right)\cdot |A_u||Y||A_v| - \delta_2(\ell)|X||Y|Z| \nonumber                                                                       \\
			               & \overset{\eqref{eq:Au-Av-size}}{\geq} \left(\frac{d_3d_2^3}{2}\right)\cdot \left(\frac{\eta}{4 v(t/2)}\right) \cdot \left(\frac{d_3d_2^3}{8}\right) |X||Y||Z| - \delta_2(\ell)|X||Y|Z| \nonumber \\
			               & \overset{\phantom{\eqref{eq:H-density}}}{=} \left(\frac{d^2_3d_2^6 \eta}{64 v(t/2)} - \delta_2(\ell) \right) \cdot |X||Y||Z| \nonumber                                                           \\
			               & \overset{\eqref{eq:delta2-r-func}}{\geq} \left(\frac{d^2_3d_2^6 \eta}{65 v(t/2)}\right) \cdot |X||Y||Z| \label{eq:Au-Av}.
		\end{align}

		If red is a majority colour seen along $E_H(A_u,Y,A_v)$, then there exists a vertex $v' \in A_v \subseteq Z$ satisfying
		$$
			\left|E\left(L_H^{(r)}(v')\right) \cap (A_u \times Y)\right| \overset{\eqref{eq:Au-Av}}{\geq}  \left(\frac{d^2_3d_2^6 \eta}{130 v(t/2)}\right) |X||Y| \geq \left(\frac{d^2_3d_2^6 \eta}{130 v(t/2)}\right) |A_u||Y|.
		$$
		Consequently, the set
		$$
			A_{u,v'} = \left\{x \in A_u\colon \deg_{L_H^{(r)}(v')}(x) \geq t\right\} \subseteq A_u \subseteq X
		$$
		satisfies
		\begin{align*}
			|A_{u,v'}| & \overset{\phantom{\eqref{eq:Au-Av-size}}}{\geq} \left(\frac{d^2_3d_2^6 \eta}{260 v(t/2)}\right) |A_u|                                                  \\
			           & \overset{\eqref{eq:Au-Av-size}}{\geq} \left(\frac{d^2_3d_2^6 \eta}{260 v(t/2)}\right) \cdot  \left( \frac{d_3 d_2^3}{8} \right) |X|                    \\
			           & \overset{\phantom{\eqref{eq:Au-Av-size}}}{\geq} \left(\frac{d^3_3 d_2^9 \eta}{2100 v(t/2)}\right) \cdot \left \lfloor \frac{n}{t'}\right \rfloor \\
			           & \overset{\eqref{eq:mu}}{\geq} \mu n,
		\end{align*}
		where the first inequality holds by Fact~\ref{cor:support}.
		We may then write $\Gamma[A_{u,v'}] \to (\tilde K^{(3)}_{t-1})$ owing to $R$ being $(\tilde K^{(3)}_{t-1},\mu)$-Ramsey, by Property~(C.3). Let~$K$ be a copy of $\tilde K^{(3)}_{t-1}$ appearing monochromatically under $\psi$ within $\Gamma[A_{u,v'}]$. Let $x_1, \ldots, x_{t-1}$ denote the branch vertices of~$K$. It follows by the definition of $A_{u,v'}$ that there are distinct vertices $y_1, \ldots, y_{t-1} \in Y$ such that $\{x_i, y_i ,v'\}$ is a red edge of $H$ for every $i \in [t-1]$. Similarly, since $A_{u,v'} \subseteq A_u$, there are distinct vertices $y'_1, \ldots, y'_{t-1} \in Y$ such that $\{x_i, y'_i ,u\}$ is a blue edge of $H$ for every $i \in [t-1]$. Therefore, if $K$ is red, then it can be extended into a red copy of $\tilde K^{(3)}_t$ including $v'$; if, on the other hand, $K$ is blue, then it can be extended into a blue copy of $\tilde K^{(3)}_t$ including $u$. In either case, a contradiction to the assumption that $\psi$ admits no monochromatic copies of~$\tilde K^{(3)}_t$ is reached.

		It remains to consider the complementary case where blue is the  majority colour in $E_H(A_u,Y,A_v)$. The argument in this case parallels that seen in the previous one with the sole cardinal difference being that instead of finding a monochromatic copy of $\tilde K^{(3)}_{t-1}$ in a subset of $A_u \subseteq X$, such a copy is found in a subset of $A_v \subseteq Z$. An argument for this case is provided for completeness. If blue is a majority colour seen along $E_H(A_u,Y,A_v)$, then there exists a vertex $u' \in A_u \subseteq X$ satisfying
		$$
			\left|E\left(L_H^{(b)}(u')\right) \cap (Y \times A_v)\right| \overset{\eqref{eq:Au-Av}}{\geq}  \left(\frac{d^2_3d_2^6 \eta}{130 v(t/2)}\right) |Y| |Z| \geq \left(\frac{d^2_3d_2^6 \eta}{130 v(t/2)}\right) |Y| |A_v|.
		$$
		Consequently, the set
		$$
			A_{v,u'} = \left\{z \in A_v\colon \deg_{L_H^{(b)}(u')}(z) \geq t\right\} \subseteq A_v \subseteq Z
		$$
		satisfies
		\begin{align*}
			|A_{v,u'}| & \overset{\phantom{\eqref{eq:Au-Av-size}}}{\geq} \left(\frac{d^2_3 d_2^6 \eta}{260 v(t/2)}\right) |A_v|                                                    \\
			           & \overset{\eqref{eq:Au-Av-size}}{\geq} \left(\frac{d^2_3d_2^6 \eta}{260 v(t/2)}\right) \cdot  \left( \frac{\eta}{4 v(t/2)} \right) |Z|                     \\
			           & \overset{\phantom{\eqref{eq:Au-Av-size}}}{\geq} \left(\frac{d^2_3d_2^6 \eta^2}{1100 v(t/2)^2}\right) \cdot \left \lfloor \frac{n}{t'}\right \rfloor \\
			           & \overset{\eqref{eq:mu}}{\geq} \mu n,
		\end{align*}
		where the first inequality holds by Fact~\ref{cor:support}.
		Then, $\Gamma[A_{v,u'}] \to (\tilde K^{(3)}_{t-1})$ owing to $R$ being $(\tilde K^{(3)}_{t-1},\mu)$-Ramsey, by Property~(C.3). A monochromatic copy of $\tilde K^{(3)}_{t-1}$ appearing in $\Gamma[A_{v,u'}]$ can be either extended into a red copy of $\tilde K^{(3)}_t$ including the vertex $v$ or into a blue such copy including $u'$. In either case, a contradiction to the assumption that $\psi$ admits no monochromatic copy of $\tilde K^{(3)}_t$ is reached and this concludes the proof of Claim~\ref{clm:blue-maj}.\phantom\qedhere\hfill{$\blacksquare$}
	\end{proof}

	The following counterpart of Claim~\ref{clm:blue-maj} holds as well.

	\begin{claim}\label{clm:red-maj}
		If red is a majority colour of $\psi$ in $H$, then $e\big(L^{(b)}_H(v)\big) \leq \frac{\eta}{2v(t/2)} \cdot |Y||Z|$ holds for every $v \in X$.\hfill{$\blacksquare$}
	\end{claim}

	Proceeding with the proof of Theorem~\ref{thm:main-Ramsey}, assume first that blue is a majority colour of $\psi$ in $H$. By Property~(C.1), either there is a red copy of $\tilde K^{(3)}_t$ (within $X$) or there is a blue copy of $\tilde K^{(3)}_{t/2}$ within $X$ not supported on $\mathcal{F}$. If the former occurs, then the proof concludes. Assume then that $K \subseteq \Gamma[X]$ is a blue copy of $\tilde K^{(3)}_{t/2}$ such that $V(K) \notin \mathcal{F}$, and write $L_H(K,P) =  \bigcap_{x \in V(K)}L_H(x, P)$ to denote the joint link graph of the members of $V(K)$ supported on $P$. Then,
	$$
		e(L_H(K,P)) \geq \left(d_3^{v(t/2)}-\eps\right) d_2^{2v(t/2)+1}|Y||Z|,
	$$
	holds by~\eqref{eq:F}.
	Remove $E(L_H^{(r)}(x))$ from $E(L_H(K,P))$ for every
	$x \in V(K)$; that is, remove any edge in $L_H(K,P)$ that together with a vertex of $K$ gives rise to a red edge of $H$ with respect to $\psi$. By Claim~\ref{clm:blue-maj}, at most
	$$
		\sum_{x \in V(K)} e\left(L_H^{(r)}(x)\right) \leq v(t/2) \cdot \frac{\eta}{2v(t/2)}|Y||Z| = \frac{\eta}{2}|Y||Z|
	$$
	edges are thus discarded from $L_H(K,P)$, leaving at least
	\begin{align*}
		\Big(\big(d_3^{v(t/2)}-\eps\big) d_2^{2v(t/2)+1}-\frac{\eta}{2}\Big)|Y||Z| 
		& \overset{\eqref{eq:d3-eps}}{\geq} 
		\Big(\frac{d_3^{v(t/2)}d_2^{2v(t/2)+1}}{2} -\frac{\eta}{2}\Big)|Y||Z| \\
		& \overset{\eqref{eq:aux}}{=} 
		\Big(\eta - \frac{\eta}{2}\Big) |Y||Z| \\
		& \overset{\phantom{\eqref{eq:aux}}}{=} 
		\frac{\eta}{2} |Y||Z|
	\end{align*}
	edges in the {\it residual} joint link graph of $K$, denoted $L'_H(K,P)$. It follows by Lemma~\ref{lem:supersaturation} and Fact~\ref{obs:disjoint-copies} that $L'_H(K,P)$ contains at least
	$$
		\xi_{\ref{obs:disjoint-copies}}(\zeta_{\ref{lem:supersaturation}}(\eta/2))\frac{2n}{T_{\ref{lem:ghrl}}} \overset{\eqref{eq:mu}}{\geq} \mu n
	$$
	vertex  disjoint copies of the bipartite graph $K_{1,t/2}$. Let $S \subseteq V(L'_H(K,P))$ consist of the centre vertices of all said copies of $K_{1,t/2}$. Property~(C.4) coupled with $|S| \geq \mu n$ collectively assert that 
\[
	\Gamma[S] \to (\tilde K^{(3)}_t,\tilde K^{(3)}_{t/2})\,.
\] 
If the first alternative occurs, then there is a red copy of $\tilde K^{(3)}_t$ and thus the proof concludes. Suppose then that the second alternative takes place so that a blue copy $K'$ of $\tilde K^{(3)}_{t/2}$ arises in~$\Gamma[S]$. Let $u_1, \ldots, u_{t/2}$ denote the branch vertices of $K'$ and let $x_1, \ldots, x_{t/2}$ denote the branch vertices of $K$. It follows by the definitions of $L'_H(K,P)$ and $S$ that there are $t^2/4$ distinct vertices $\{w_{ij} \colon i, j \in [t/2]\} \subseteq V(L'_H(K,P)) \setminus \{u_1, \ldots, u_{t/2}, x_1, \ldots, x_{t/2}\}$ such that $\{u_i, x_j, w_{ij}\}$ forms a blue edge of $H$ for every $i, j \in [t/2]$. We conclude that $\Gamma$ admits a copy of $\tilde K^{(3)}_t$ which is blue under $\psi$.

	Next, assume that red is a majority colour seen for $\psi$ in $H$. Replacing the appeals to Claim~\ref{clm:blue-maj}, Properties~(C.1) and~(C.4) in the argument above with appeals to Claim~\ref{clm:red-maj} and Properties~(C.2), and~(C.5), respectively, leads to the rise of a monochromatic copy of $\tilde K^{(3)}_t$ in~$\Gamma$ under $\psi$ in this case as well.
\end{proof}

\section{Proof of the tuple lemma}\label{sec:tuple}

The tuple lemma (Proposition~\ref{lem:tuple}) follows from a straightforward application of the Cauchy--Schwarz inequality and the counting lemma of Nagle and R\"odl~\cite{NR03}*{Theorem~9.0.2}.
For a fixed integer $t\geq 2$, let $K^{(3)}_{t,1,1}$ denote the complete $3$-partite $3$-uniform hypergraph with one vertex class of order $t$ and the other two classes consisting only of a single vertex each. Write~$D_t$ to denote the hypergraph obtained from taking two copies 
of $K^{(3)}_{t,1,1}$ and identifying each of the $t$ vertices in the first class of the first copy with the corresponding vertex in the second copy. It follows from the aforementioned counting lemma that in a sufficiently regular triad, the number 
of copies of $K^{(3)}_{t,1,1}$ and the number of copies of $D_t$ are as `expected'; the Cauchy--Schwarz inequality then implies the conclusion of the tuple lemma for joint links. A more direct proof of the tuple lemma for joint links, which does not rely on the counting lemma, can be found in 
Appendix~\ref{sec:direct-tuple}.

\begin{proof}[Proof of Proposition~\ref{lem:tuple}]
	Given $t\geq 2$ and $\eps$, $d_3>0$, let $\delta_3 >0$ be sufficiently small so as to render the aforementioned counting lemma applicable for $K^{(3)}_{t,1,1}$ and $D_t$ with relative error $\gamma=\eps^3/4$. For a given $d_2>0$, an additional appeal to the counting lemma delivers $\delta_2$ and $r$. Let $H=(X\dcup Y\dcup Z, E_H)$ and $P=(X\dcup Y\dcup Z, E_p)$ satisfy the assumption of Proposition~\ref{lem:tuple} and assume, without loss of generality, that $E_H\subseteq \mathcal{K}_3(P)$.
Write $\mathfrak{K}_t$ to denote the number of injective homomorphisms of $K^{(3)}_{t,1,1}$  in  $H$ with the $t$ vertices of the first class contained in $X$ and, similarly, write 
	$\mathfrak{D}_t$ to denote the number of injective homomorphisms of $D_t$ in $H$.
	
	With the above definitions in place, observe that
	\[
		\sum_{(x_1,\dots,x_t)\in X^t}\big|L_H(\{x_1,\dots,x_t\},P)\big|\geq \mathfrak{K}_t
	\]
	as well as
	\[
		\sum_{(x_1,\dots,x_t)\in X^t}\big|L_H(\{x_1,\dots,x_t\},P)\big|^2\leq \mathfrak{D}_t-2|X|^t|Y||Z|\big(|Y|+|Z|\big)
		\,
	\]
	both hold, where the lower order error term in the last inequality accounts for `degenerate copies $D_t$' identifying both vertices in $Y$ or $Z$, 
	which are counted in the sum of squares.
	
	The counting lemma applied to $K^{(3)}_{t,1,1}$ and $D_t$ in $H$ with respect to $P$ asserts that
	\[
		\mathfrak{K}_t \geq (1-\gamma)d_3^td_2^{2t+1}|Y||Z|\cdot|X|^t
		\quad\text{and}\quad
		\mathfrak{D}_t \leq (1+\gamma)\big(d_3^td_2^{2t+1}|Y||Z|\big)^2\cdot |X|^t\,.
	\]
	For sufficiently large sets $X$, $Y$, and $Z$,  
	\[
		2|X|^t|Y||Z|\big(|Y|+|Z|\big) \leq \gamma  \big(d_3^td_2^{2t+1}|Y||Z|\big)^2\cdot |X|^t
	\]
	holds; consequently, we arrive at
	\begin{equation}\label{eq:t-lb}
		\sum_{(x_1,\dots,x_t)\in X^t}\big|L_H(\{x_1,\dots,x_t\},P)\big|\geq (1-\gamma)d_3^td_2^{2t+1}|Y||Z|\cdot|X|^t
	\end{equation}
	and 
	\begin{equation}\label{eq:t-ub}
		\sum_{(x_1,\dots,x_t)\in X^t}\big|L_H(\{x_1,\dots,x_t\},P)\big|^2\leq (1+2\gamma)\big(d_3^td_2^{2t+1}|Y||Z|\big)^2\cdot |X|^t\,.
	\end{equation}
	Finally, for $\mathfrak{B}$ being the number of all those `bad' $t$-tuples $(x_1,\dots,x_t)\in X^t$ 
	with 
	\[
		\Big|\big|L_H(\{x_1,\dots,x_t\},P)\big|-d_3^td_2^{2t+1}|Y||Z|\Big|>\eps d_3^td_2^{2t+1}|Y||Z|, 
	\]
	we infer
	\[
		\mathfrak{B}\cdot \big(\eps d_3^td_2^{2t+1}|Y||Z|\big)^2
		\leq 
		\sum_{(x_1,\dots,x_t)\in X^t}\Big(\big|L_H(\{x_1,\dots,x_t\},P)\big|-d_3^td_2^{2t+1}|Y||Z|\Big)^2\,.
	\]
	Expanding the quadratic expression
	and applying the estimates from~\eqref{eq:t-lb} and~\eqref{eq:t-ub} finally yields
	\[
		\mathfrak{B}\cdot \big(\eps d_3^td_2^{2t+1}|Y||Z|\big)^2
		\leq 
		4\gamma\cdot\big(d_3^td_2^{2t+1}|Y||Z|)^2\cdot|X|^t
	\]
	and the proposition follows from the choice of $\gamma=\eps^3/4$.  
\end{proof}

\section{A variant of the regularity lemma for hypergraphs}\label{sec:ghrl}

In this section, we prove Proposition~\ref{lem:ghrl} which is our new variant of the Strong Lemma (Lemma~\ref{lem:shrl}). In section~\ref{sec:blueprint}, we lay out a `blueprint' for our proof of Proposition~\ref{lem:ghrl} and, in the course of which, collect all results from~\cite{SRL} facilitating our proof. This `blueprint' is then carried out in Section~\ref{sec:HRL-proof}, where a detailed proof of Proposition~\ref{lem:ghrl} is provided.

\subsection{Outline of the proof}\label{sec:blueprint}
The proof of Proposition~\ref{lem:ghrl} follows the lines of Szemer\'edi's proof of the regularity lemma for graphs~\cite{Sz78}. In particular, it is based on an index increment argument along refining partitions.  More precisely, we adapt the proof of Lemma~\ref{lem:shrl} from~\cite{SRL}. Below we present a detailed outline and along those lines we state several lemmata from~\cite{SRL}, which we shall employ in the proof.

\subsubsection*{Refinement process} We start by providing an {\it initial} pair of partitions (defined below), one over vertices and the other over (some) pairs of vertices, satisfying Properties~(R.1-3).
If these satisfy Property~(R.4) as well, then the proof concludes. Otherwise, a {\it refinement process} for these two partitions commences. A single iteration of this process accepts as input a pair of partitions $\Pi = (\mathcal{V},\mathcal{B})$, where $\mathcal{V}$ is a vertex  partition and $\mathcal{B}$ is a partition of $K^{(2)}(\mathcal{V})$, such that~$\Pi$ satisfies Properties~(R.1-3) but not~(R.4). At the end of the iteration, a pair of partitions $\Pi' = (\mathcal{V}',\mathcal{B}')$ satisfying Properties~(R.1-3) is produced such that $\Pi' \prec \Pi$, by which we mean that $\mathcal{V}' \prec \mathcal{V}$ and $\mathcal{B}' \prec \mathcal{B}$.

The pair $\Pi'$ has an additional crucial property. As customary in proofs of regularity lemmata, a quantity called the {\it index} (see~\eqref{eq:index}) is associated with any pair of partitions; in that, certain quantities, namely $\mathrm{ind}(\Pi)$ and $\mathrm{ind}(\Pi')$, are associated with $\Pi$ and $\Pi'$, respectively. The additional key property satisfied by $\Pi'$, alluded to above, is that $\mathrm{ind}(\Pi') \geq \mathrm{ind}(\Pi) + \Omega_{\delta^{(3)}_{\ref{lem:ghrl}}}(1)$; this inequality embodies the traditional {\it index increment} argument that often appears in proofs of regularity lemmata.

If the pair $\Pi'$ satisfies Property~(R.4), then the proof concludes; otherwise another iteration of the refinement process takes place, this time with $\Pi'$ assuming the role of $\Pi$ above. The index increment argument and the fact that the index of any pair of partitions is bounded from above by one (see~\eqref{eq:index-bound}), imply that such a refinement process must terminate. Therefore, within $O(1)$ iterations, a pair of partitions satisfying Properties~(R.1-4) is encountered and Proposition~\ref{lem:ghrl} is proved. Figure~\ref{fig:process} provides a bird's eye view of a single iteration of the refinement process.

\subsubsection*{Initial partitions} The first pair of partitions, namely $\mathcal{V}_0$ and $\mathcal{B}_0$, from which the proof of Proposition~\ref{lem:ghrl} commences is defined next. Let $V$ denote the common vertex  set of $H_1,\ldots,H_{s_{\ref{lem:ghrl}}}$ and let $\mathcal{V}_0$ be an equitable vertex  partition of the form $V = V_1 \discup \cdots \discup V_{t}$, with $t$ some positive integer, such that $H_i$ is $\delta^{(2)}_{\ref{lem:ghrl}}(\ell_{\ref{lem:ghrl}})$-weakly regular with respect to $\mathcal{V}_0$ for every $i \in [s_{\ref{lem:ghrl}}]$. Such a partition exists by the Weak Lemma (Lemma~\ref{lem:weak}) applied with
\begin{equation}\label{eq:initial-t}
	t_{\ref{lem:weak}} \gg \max \left\{ (\delta^{(3)}_{\ref{lem:ghrl}})^{-4}, t_{\ref{lem:ghrl}} \right\},
\end{equation}
$\delta_{\ref{lem:weak}} = \delta^{(2)}_{\ref{lem:ghrl}}(\ell_{\ref{lem:ghrl}})$, the trivial partition $\mathcal{U}_{\ref{lem:weak}} = V$ (i.e., $h_{\ref{lem:weak}}=1$), and the given sequence $H_1,\ldots,H_{s_{\ref{lem:ghrl}}}$. In preparation for a subsequent application of Lemma~\ref{lem:index-inc} (stated below), one may further assume that $t$ is sufficiently large so as to ensure that $e(K^{(3)}(\mathcal{V}_0)) \geq (1-\delta^{(3)}_{\ref{lem:ghrl}}/2) \binom{|V|}{3}$ holds (and thus have~\eqref{eq:edges}, stated below, satisfied).

Let $\mathcal{B}_0$ be the partition of $K^{(2)}(\mathcal{V}_0)$ defined as follows. For every $1 \leq i < j \leq t$, let $B^{ij}_1, \ldots, B^{ij}_{\ell_{\ref{lem:ghrl}}}$ be a uniform random edge colouring of $K^{(2)}(V_i,V_j)$ using $\ell_{\ref{lem:ghrl}}$ colours. That is, every edge of $K^{(2)}(V_i,V_j)$ is assigned a colour from $[\ell_{\ref{lem:ghrl}}]$ uniformly at random and independently of all other edges of $K^{(2)}(V_i,V_j)$. For every $k \in [\ell_{\ref{lem:ghrl}}]$, every pair of indices $1 \leq i < j \leq t$, and any pair of subsets $X \subseteq V_i$ and $Y \subseteq V_j$, it holds that
$$
	\Ex\left[e\left(B^{ij}_k[X,Y]\right)\right] = \frac{1}{\ell_{\ref{lem:ghrl}}}|X||Y|\,.
$$
Applying Chernoff's inequality~\cite{JLR}*{Theorem~2.1} yields
$$
	\Pr\Big(
	\Big|e\big(B^{ij}_k[X,Y]\big)  - \frac{1}{\ell_{\ref{lem:ghrl}}}|X||Y| \Big|
	>
	\delta^{(2)}_{\ref{lem:ghrl}}(\ell_{\ref{lem:ghrl}})|V_i||V_j| \Big)
	=
	\exp(-\Omega(n^2))\,,
$$
where here we rely on $\mathcal{V}$ being equitable, implying  that $|V_i|,|V_j| = \Omega(n)$. A union-bound over all choices of $k, i, j, X$, and $Y$ implies that a.a.s.\ $B^{ij}_k$ is $(\delta^{(2)}_{\ref{lem:ghrl}}(\ell_{\ref{lem:ghrl}}),\ell_{\ref{lem:ghrl}}^{-1})$-regular for every $k \in [\ell_{\ref{lem:ghrl}}]$ and every pair of indices $1 \leq i < j \leq t$. In particular, $\mathcal{B}_0$ is $\ell_{\ref{lem:ghrl}}$-equitable.

\subsubsection*{Index increment} The notion of {\it index}, employed in our arguments, is defined next, along with the index increment machinery alluded to above. Let $V$ be a finite set and let $\mathcal{V}$ be the partition given by $V= V_1 \discup \ldots \discup V_h$, with $h \geq 1$ some integer. Let $\mathcal{B}$ be an $\ell$-equitable partition with respect to $\mathcal{V}$, for some integer $\ell \geq 1$. The {\em index} of $\mathcal{B}$ with respect to $\mathcal{V}$ and a partition $\mathcal{H}$ of $V^{(3)}$
is given by
\begin{equation}\label{eq:index}
	\mathrm{ind}(\mathcal{B}) = \frac{1}{|V|^3} \sum_{H \in \mathcal{H}} \sum_{P} d\left(H|P\right)^2\cdot|\mathcal{K}_3(P)|,
\end{equation}
where the second sum ranges over the triads of $\mathcal{B}$. It is easy to check (see, e.g.,~\cite{SRL}*{Fact~33}) that
\begin{equation}\label{eq:index-bound}
	\mathrm{ind}(\mathcal{B}) \in [0,1].
\end{equation}
The notion of the index, seen in~\eqref{eq:index}, fits the case $s_{\ref{lem:ghrl}}=1$, i.e., the case in which a single hypergraph $H$ is to be regularised. In this case, the members of the partition $\mathcal{H}$, seen in~\eqref{eq:index}, are $H$ and its complement; this is in accordance with~\cite{SRL} (see the implication of~\cite{SRL}*{Theorem~17}
from~\cite{SRL}*{Theorem~23}). Our formulation of Proposition~\ref{lem:ghrl} supports $s_{\ref{lem:ghrl}} > 1$; in the terminology of~\cite{KSSS}, it is a {\it multi-colour regularity lemma}.

To support the multi-colour version, the standard approach (see, e.g.,~\cite{KSSS}) is to define the above index for each hypergraph (and its complement) being regularised and then define a new version of the index given by the average of all of the aforementioned indices of the individual hypergaphs (taking the average ensures that the new index is upperbounded by one as well).

\begin{remark}\label{re:s=1}
	As the transition to the multi-coloured version is considered standard, we prove Proposition~\ref{lem:ghrl} under the assumption that $s_{\ref{lem:ghrl}}=1$. Our proof of Theorem~\ref{thm:main-Ramsey} does not employ the multi-colour version.
\end{remark}

We shall employ the following index increment lemma~\cite{SRL}*{Proposition~39}.
\begin{lemma}[Index increment lemma]\label{lem:index-inc}
	Let $V$ be a finite set, let $\mathcal{V}$ be a partition of $V$, let $\mathcal{H}$ be a partition of $V^{(3)}$, and  let $\mathcal{B}$ be an $\ell$-equitable partition of $K^{(2)}(\mathcal{V})$. Furthermore, let an integer $r = r_{\ref{lem:index-inc}} \geq 1$ be given and let $\delta = \delta_{\ref{lem:index-inc}} >0$ satisfy
	\begin{equation}\label{eq:edges}
		e\left(K^{(3)}(\mathcal{V})\right) \geq (1-\delta/2)\binom{|V|}{3}.
	\end{equation}
	If there exists an $H \in \mathcal{H}$ that is $(\delta,r)$-irregular with respect to $\mathcal{B}$, then there exists a partition~$\mathcal{B}'$ of $V^{(2)}$ satisfying
	\begin{description}
		\item [(INC.1)] $\mathcal{B}' \prec \mathcal{B}$,
		\item [(INC.2)] $|\mathcal{B}'| \leq |\mathcal{B}| \cdot 2^{r |\mathcal{V}| \ell^2} \leq |\mathcal{V}|^2 \ell \cdot 2^{r |\mathcal{V}| \ell^2}$, and
		\item [(INC.3)] $\mathrm{ind}(\mathcal{B}') \geq \mathrm{ind}(\mathcal{B}) +\delta^4/2$.\qed
	\end{description}
\end{lemma}

The partitions $\mathcal{B}$ and $\mathcal{B'}$, appearing in the premise of Lemma~\ref{lem:index-inc}, are taken over $K^{(2)}(\mathcal{V})$ and $V^{(2)}$, respectively. In Section~\ref{sec:HRL}, it is explained what is meant by stating that a partition of $V^{(2)}$ refines a partition of $K^{(2)}(\mathcal{V})$.

Returning to the general scheme of the proof outline, let $(\mathcal{V}_1,\mathcal{B}_1)$ be a pair of partitions from which a single iteration of the refinement process commences; these partitions satisfy Properties~(R.1-3) but they do {\it not} satisfy Property~(R.4). The Index Increment Lemma (Lemma~\ref{lem:index-inc}) applied to $(\mathcal{V}_1,\mathcal{B}_1)$ produces a pair $(\mathcal{V}_1,\mathcal{B}'_1)$ with $\mathcal{B}_1$ and $\mathcal{B}'_1$ satisfying~(INC.1-3). That is, $\mathcal{B}'_1$ refines $\mathcal{B}_1$, its size is $O(|\mathcal{B}_1|) = O(1)$, and most importantly satisfies
\[
	\mathrm{ind}(\mathcal{B}'_1) \geq \mathrm{ind}(\mathcal{B}) + \Omega_{\delta^{(3)}_{\ref{lem:ghrl}}}(1)\,.
\]
While $\mathcal{B}'_1$ has its index elevated appropriately relative to $\mathcal{B}_1$, it may have lost Property~(R.2).
Indeed, $\mathcal{B}'_1$ arises from considering the Venn diagram of all witnesses of $(\delta^{(3)}_{\ref{lem:ghrl}},r_{\ref{lem:ghrl}})$-irregularity that the sequence $(H_1,\ldots,H_{s_{\ref{lem:ghrl}}})$ has across $(\mathcal{V}_1,\mathcal{B}_1)$; see~\cite{SRL} for further details. The process then continues with further refinements of $(\mathcal{V}_1,\mathcal{B}'_1)$ so as to regain Property~(R.2).

\subsubsection*{Approximation} The approximation lemma~\cite{SRL}*{Lemma~25} serves as a key tool through which we regain Property~(R.2). It can be viewed as a generalisation of a result from~\cite{AFKZ00}, which handles the corresponding graph case.

\begin{lemma}[Approximation lemma]\label{lem:approx}
	For every pair of integers $s = s_{\ref{lem:approx}} \geq 1$ and $h = h_{\ref{lem:approx}} \geq 1$, every real number $\nu = \nu_{\ref{lem:approx}} >0$, and every function $\eps= \eps_{\ref{lem:approx}}\colon \mathds{N} \to (0,1]$, there exist positive integers $t_{\ref{lem:approx}}$ and $n_0$ such that the following holds whenever $n \geq n_0$. Let $V$ be a finite set of size $n$, let $\mathcal{V}$ be a partition of the form $V=V_1 \discup \cdots \discup V_h$ with all its parts having size~$\Omega(n)$, and let $\mathcal{B} = (B_1,\ldots,B_s)$ be a partition of $V^{(2)}$. Then, there exists an equitable partition $\{V = U_1 \discup \cdots \discup U_{t_{\ref{lem:approx}}}\} \prec \mathcal{V}$, namely $\mathcal{U}$, as well as a partition $\mathcal{B}' = (B'_1,\ldots,B'_s)$ of $V^{(2)}$ such that the following holds.
	\begin{description}
		\item [(APX.1)] $B'_k[U_i,U_j]$ is $
			      \eps(t_{\ref{lem:approx}})$-regular for every $k \in [s]$
		      and every $1 \leq i < j \leq t_{\ref{lem:approx}}$.

		\item [(APX.2)] $|E(B_i) \triangle E(B'_i)| \leq \nu n^2$ for
		      every $i \in [s]$.\qed
	\end{description}
\end{lemma}

 The approximation lemma bares its name as it replaces every $B \in \mathcal{B}$ with a highly regular bipartite graph $B'$ of virtually the same density as $B$ (as specified in (APX.2)). However, $B'$ need {\it not} be a subgraph of $B$ and may contain edges not present in $B$; the latter degrades the index increment attained by the Index Increment Lemma (Lemma~\ref{lem:index-inc}).

In its original formulation, namely~\cite{SRL}*{Lemma~25}, the approximation lemma also entails a divisibility condition which in our formulation would read as $t_{\ref{lem:approx}}! \mid n$. The term $t_{\ref{lem:approx}}!$ is a fixed constant. In the case that $t_{\ref{lem:approx}}! \nmid n$ holds, our formulation takes into account any degradation of all parameters that may be incurred by having to distribute at most $t_{\ref{lem:approx}}!$ vertices amongst the members of $\mathcal{U}$.

Returning to the pair $(\mathcal{V}_1,\mathcal{B}'_1)$ and its potential loss of Property~(R.2), the approximation lemma (Lemma~\ref{lem:approx}) is applied to this pair so as to produce a pair $(\mathcal{V}_2,\mathcal{B}_2)$ such that $\mathcal{V}_2 \prec \mathcal{V}_1$ and $\mathcal{B}_2$ approximates $\mathcal{B}'_1$ per~(APX.2). That is, modulo some $\xi n^2$ exceptional pairs, with $\xi>0$ being arbitrarily small yet fixed, the partition $\mathcal{B}_2$ refines $\mathcal{B}'_1$. 

Utilising the fact that the members of $\mathcal{B}_2$ are highly regular, per~(APX.1), we proceed to {\it randomly} slice (see Steps~I.a and~I.b in the proof of Proposition~\ref{lem:ghrl} for details) the members of~$\mathcal{B}_2$ so as to obtain a partition~$\mathcal{B}_3$ which, modulo some $\xi' n^2$ exceptions, with $\xi' >0$ being  arbitrarily small yet fixed, refines $\mathcal{B}_2$ and such that its members satisfy Property~(R.2), in that all its members are at the `correct' density and regularity as required by Property~(R.2).
This stage of the process culminates with the pair $(\mathcal{V}_2,\mathcal{B}_3)$ and with having Property~(R.2) regained. 

Having $\mathcal{B}_2$ essentially refining $\mathcal{B}'_1$ and $\mathcal{B}_3$ essentially refining $\mathcal{B}_2$ modulo some $o(n^2)$ exceptions each time, plays a crucial part in the forthcoming index manipulation arguments appearing below. Unfortunately, due to the application of the Approximation Lemma (Lemma~\ref{lem:approx}), it is possible that the pair $(\mathcal{V}_2,\mathcal{B}_3)$ does not satisfy Property~(R.3).

\subsubsection*{Weak regularity re-established} To regain Property~(R.3), the Weak Lemma (Lemma~\ref{lem:weak}) is applied to $\mathcal{V}_2$ so as to attain a vertex  partition, namely $\mathcal{V}_3$, such that $\mathcal{V}_3 \prec \mathcal{V}_2$ and such that~$H$ is weakly regular with respect to $\mathcal{V}_3$ at the required level. This in turn affects the regularity of the members of $\mathcal{B}_3$ with respect to the members of $\mathcal{V}_3$ in the sense that the satisfaction of Property~(R.2) is again in jeopardy. Repeated applications of Lemma~\ref{lem:slicing} are then used to regain Property~(R.2) once more. The key point at this stage is that the degradation in the regularity of the members of $\mathcal{B}_3$, with respect to $\mathcal{V}_3$, can be anticipated prior to the application of Lemma~\ref{lem:approx} which in turn allows for an application of the latter with an enhanced regularity threshold, in order to compensate for this eventual degradation.

This stage ends with a pair $(\mathcal{V}_3,\mathcal{B}_4)$, where $\mathcal{B}_4 \prec \mathcal{B}_3$ and is the result of the aforementioned repeated applications of Lemma~\ref{lem:slicing}. Accurate details regarding this stage can be seen in Step~II of the proof of Proposition~\ref{lem:ghrl}.

\subsubsection*{Index manipulations} Tracking the index of the various partitions encountered throughout the process described above, we start with the inequality 
\[
	\mathrm{ind}(\mathcal{B}'_1) \geq \mathrm{ind}(\mathcal{B}_1) + \Omega_{\delta^{(3)}_{\ref{lem:ghrl}}}(1)
\]
supported by the index increment lemma (Lemma~\ref{lem:index-inc}). From here on out this index increment suffers degradation incurred by the refinement process outlined above. Two tools are used to curb this degradation. The first such tool 
from~\cite{SRL}*{Proposition~34}, stated below in Lemma~\ref{lem:index-approx}, is designed to handle the index of partitions produced by the Lemma~\ref{lem:approx}. The second is Lemma~\ref{lem:refine} below, which provides estimates for the index of all partitions encountered with the exception of the one produced by the approximation lemma (Lemma~\ref{lem:approx}).

\begin{lemma}\label{lem:index-approx}
	Let $s= s_{\ref{lem:index-approx}}, h= h_{\ref{lem:index-approx}}$, and $t = t_{\ref{lem:index-approx}} \geq h$ be positive integers and let $\nu = \nu_{\ref{lem:index-approx}} > 0$. Let $V$ be a set of size $n$ and let $\mathcal{V} = V_1 \discup \cdots \discup V_h$ and $\mathcal{U} = U_1 \discup \cdots \discup U_t$ be partitions of $V$ such that $\mathcal{U} \prec \mathcal{V}$. Let $\mathcal{H}$ be a partition of $V^{(3)}$ and
	let $\mathcal{B} = (B_1,\ldots,B_s)$ and $\mathcal{B}' = (B'_1,\ldots,B'_s)$ be partitions of $K^{(2)}(\mathcal{V})$ and $K^{(2)}(\mathcal{U})$, respectively, satisfying
	$
		|E(B_i) \triangle E(B'_i)| \leq \nu n^2
	$
	for every $i \in [s]$. Then,
	\begin{equation}\label{eq:ind-approx-new}
		\mathrm{ind}(\mathcal{B'}) \geq \mathrm{ind}(\mathcal{B}) - 9(2s)^3 \nu
	\end{equation}
	holds, where $\mathrm{ind}(\mathcal{B})$ is taken with respect to $\mathcal{V}$ and $\mathcal{H}$, and $\mathrm{ind}(\mathcal{B'})$ is taken with respect to $\mathcal{U}$ and $\mathcal{H}$.\qed
\end{lemma}

The partitions $\mathcal{B}$ and $\mathcal{B}'$ defined in Lemma~\ref{lem:index-approx} can both be viewed as partitions 
of $V^{(2)}$ (see Section~\ref{sec:HRL}) so as to fit the formulation of the approximation lemma (Lemma~\ref{lem:approx}) and of~\cite{SRL}*{Proposition~34}. The latter, fits more general settings than the one appearing in its adaptation stated in Lemma~\ref{lem:index-approx}.

The partition~$\mathcal{B}_2$ is produced by the Approximation Lemma (Lemma~\ref{lem:approx}) applied to $(\mathcal{V}_1,\mathcal{B}'_1)$. Lemma~\ref{lem:index-approx} coupled with a judicious choice of $\nu_{\ref{lem:approx}}$ made upon applying the Approximation Lemma (and prior to the application of the Index Approximation Lemma), yields
$$
	\mathrm{ind}(\mathcal{B}_2) \geq \mathrm{ind}(\mathcal{B}'_1) - \nu'_{\ref{lem:approx}} \geq \mathrm{ind}(\mathcal{B}_1) + \Omega_{\delta^{(3)}_{\ref{lem:ghrl}}}(1) - \nu'_{\ref{lem:approx}},
$$
where $\nu'_{\ref{lem:approx}} >0$ is a constant related to $\nu_{\ref{lem:approx}}$ through~\eqref{eq:ind-approx-new}.
Being able to anticipate this degradation of the index allows for an appropriate choice of $\nu_{\ref{lem:approx}}$ to be passed to the Approximation Lemma so as to render $\mathrm{ind}(\mathcal{B}_2) \geq \mathrm{ind}(\mathcal{B}_1) + \Omega_{\delta^{(3)}_{\ref{lem:ghrl}}}(1)$.

The second tool for curbing the degradation of the index, namely the Index Approximation Lemma, is stated next. Let $\mathcal{V}$ be a partition of a finite set $V$ and let $\mathcal{B}$ and $\mathcal{B}'$ be partitions of $K^{(2)}(\mathcal{V})$. Given a non-negative real number $\beta$, a partition $\mathcal{B}'$ is said to form a $\beta$-{\em refinement} of a partition $\mathcal{B}$, denoted $\mathcal{B'} \prec_\beta \mathcal{B}$, if $\sum e(B') \leq \beta |V|^2$, where the sum is extended over 
\[
	\{B' \in \mathcal{B}' \colon B' \not\subseteq B \ \text{for every} \  B \in \mathcal{B}\}\,.
\]
The following simple lemma from~\cite{SRL}*{Proposition~38} asserts that a partition of $K^{(2)}(\mathcal{V})$ that $\beta$-refines another partition of $K^{(2)}(\mathcal{V})$ has its index at most $\beta$ `away' from that of the partition being refined.

\begin{lemma}[Index approximation lemma]\label{lem:refine}
	Let $\beta$ be a non-negative real number, let $\mathcal{V}$ be a partition of a finite set $V$, let $\mathcal{B}$ and $\mathcal{B}'$ be partitions of $K^{(2)}(\mathcal{V})$, and let $\mathcal{H}$ be a partition of~$V^{(3)}$. 
	If $\mathcal{B}' \prec_\beta \mathcal{B}$, then $\mathrm{ind}(\mathcal{B}') \geq \mathrm{ind}(\mathcal{B}) - \beta$, where here the index is taken with respect to~$\mathcal{V}$ and~$\mathcal{H}$.\qed
\end{lemma}

The next degradation in the index is incurred through the production of the partition~$\mathcal{B}_3$ from $\mathcal{B}_2$ via random slicing. We prove that $\mathcal{B}_3$ forms a $\beta$-refinement of $\mathcal{B}_2$, where $\beta > 0$ is small enough to ensure that $\mathrm{ind}(\mathcal{B}_3) \geq \mathrm{ind}(\mathcal{B}_1) + \Omega_{\delta^{(3)}_{\ref{lem:ghrl}}}(1)$ can still be inferred. The last partition, namely $\mathcal{B}_4$, properly refines $\mathcal{B}_3$, i.e. $\mathcal{B}_4 \prec_0 \mathcal{B}_3$, and thus, by the Index Approximation Lemma, no index degradation is incurred in the production of $\mathcal{B}_4$ culminating with $\mathrm{ind}(\mathcal{B}_4) \geq \mathrm{ind}(\mathcal{B}_1) + \Omega_{\delta^{(3)}_{\ref{lem:ghrl}}}(1)$.

Upon the termination of the entire refinement process, a pair of partitions $(\mathcal{V},\mathcal{B})$ satisfying Properties~(R.1-4) is produced. Setting $t = |\mathcal{V}|$ and $T_{\ref{lem:ghrl}} = |\mathcal{B}|$, yields $t \leq T_{\ref{lem:ghrl}}$ as well as $\ell \leq  T_{\ref{lem:ghrl}}$, where $\ell$ and $t$ are per the premise of Proposition~\ref{lem:ghrl}.

\begin{figure}[ht]
	\label{fig:process}
	\includegraphics[scale=0.5, width = \textwidth]{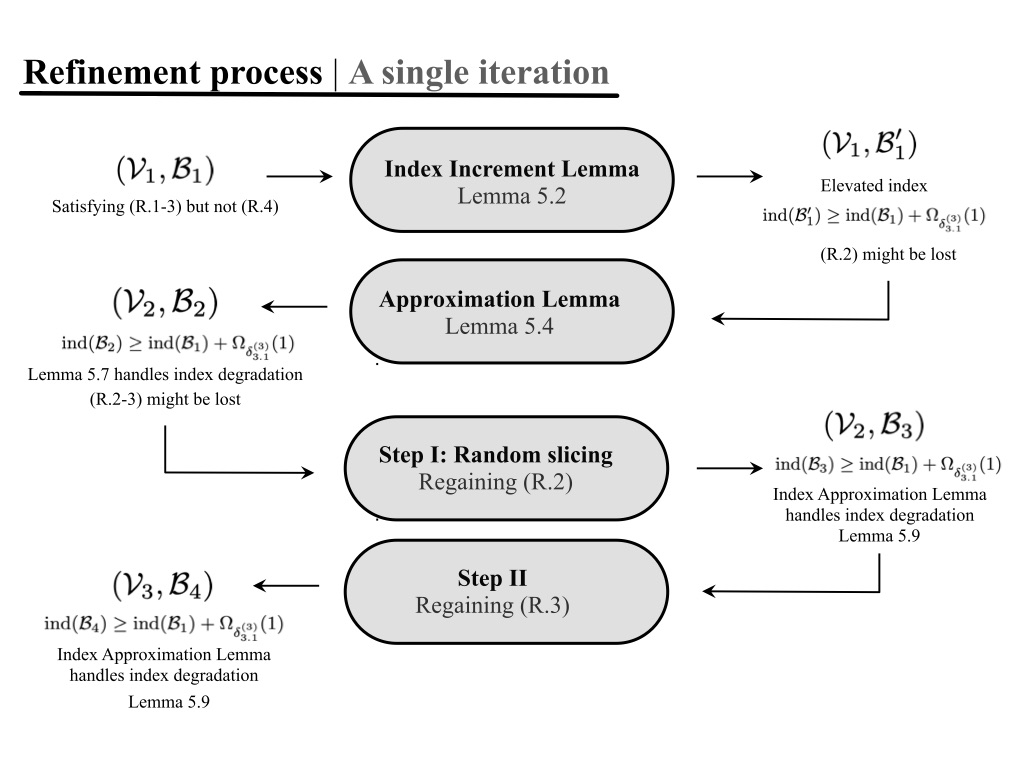}
	\caption{A single iteration of the refinement process}
\end{figure}

\subsection{Proof of the regularity lemma}\label{sec:HRL-proof}
We now give the details of the proof outlined in \ssign\ref{sec:blueprint}.
\begin{proof}[Proof of Proposition~\ref{lem:ghrl}]
Let $\delta_3 = \delta^{(3)}_{\ref{lem:ghrl}}, \delta_2 = \delta^{(2)}_{\ref{lem:ghrl}},r = r_{\ref{lem:ghrl}},\ell_{\ref{lem:ghrl}},t_{\ref{lem:ghrl}}, s_{\ref{lem:ghrl}}$ and $(H_1,\ldots,H_{s_{\ref{lem:ghrl}}})$ be as specified in the premise of Proposition~\ref{lem:ghrl}; recall Remark~\ref{re:s=1} where it is stipulated that we assume that $s_{\ref{lem:ghrl}}=1$ so that $H = H_1 = \cdots = H_{s_{\ref{lem:ghrl}}}$. It suffices to prove an index increment along a single iteration of the refinement process described in Section~\ref{sec:blueprint}. To that end, let $V$ be the vertex  set of $H$ and
let $\mathcal{V}_1$ be an equitable partition of the form $V = V_1 \discup \cdots \discup V_{t_1}$, with $t_1$ some positive integer, such that $H$ is $\delta_2(\ell_1)$-weakly regular with respect to $\mathcal{V}_1$, where~$\ell_1$ is some positive integer. Let $\mathcal{B}_1$ be an $\ell_1$-equitable partition of $K^{(2)}(\mathcal{V}_1)$. Assume that 
\[
	e(K^{(3)}(\mathcal{V}_1)) \geq (1-\delta_3/2) \binom{|V|}{3}
\] 
holds and that the partitions $\mathcal{V}_1$ and $\mathcal{B}_1$ satisfy Properties~(R.1-3) yet fail to satisfy Property~(R.4).

The Index Increment Lemma (Lemma~\ref{lem:index-inc}) applied with $\delta_{\ref{lem:index-inc}} = \delta_3$ and $r_{\ref{lem:index-inc}} = r(t_1, \ell_1)$, asserts that there exists a partition $\mathcal{B}'_1$ of $K^{(2)}(\mathcal{V}_1)$ refining $\mathcal{B}_1$ such that $\ell_2 = |\mathcal{B}'_1| \leq t^2_1 \ell_1 \cdot 2^{r(t_1, \ell_1) t_1 \ell_1^2}$ and $\mathrm{ind}(\mathcal{B}'_1) \geq \mathrm{ind}(\mathcal{B}_1) + \delta^4_3/2$. For future reference, we record that
\begin{equation}\label{eq:ell_2}
	\ell_2 \gg \delta_3^{-4}
\end{equation}
holds; this on account of $\ell_2 \geq t_1 \overset{\eqref{eq:initial-t}}{\gg} \delta_3^{-4}$.

We proceed with a two stage argument, captured below in {Step~I} and {Step~II}, through which the pair $(\mathcal{V}_1,\mathcal{B}'_1)$ is further refined so as to obtain a pair of partitions satisfying Properties~(R.1-3) and whose index is appropriately elevated with respect to that of  $\mathcal{B}_1$.

\subsubsection*{Step I: Regaining Property~(R.2)} Let $\mathcal{V}_2$ and $\mathcal{B}_2$ be the vertex and pair partitions, respectively, whose existence is guaranteed by the Approximation Lemma (Lemma~\ref{lem:approx}) applied to $\mathcal{V}_1$ and~$\mathcal{B}'_1$ with
\begin{equation}\label{eq:const}
	s_{\ref{lem:approx}} = |\mathcal{B}'_1| = \ell_2, \quad  h_{\ref{lem:approx}} = t_1, \quad \nu= \nu_{\ref{lem:approx}} \ll \frac{\delta^4_3}{s^3_{\ref{lem:approx}}}, \quad \eps_{\ref{lem:approx}} = \eps/C_1
\end{equation}
where $C_1 \gg \ell_2^2$ is some auxiliary constant and for every $t' \in \mathds{N}$
\begin{equation}\label{eq:eps-func}
	\eps(t') = \min \left\{\ell_2^{-2}, \frac{\delta_2(\ell_2^2)}{T^2_{\ref{lem:weak}}(t', \delta_2(\ell_2^2))}\right\}.
\end{equation}
The partition $\mathcal{V}_2$ refines the partition $\mathcal{V}_1$ and has the form $V = U_1 \discup U_2 \cdots \discup U_{t_2}$, where $t_2 = t_{\ref{lem:approx}}(h_{\ref{lem:approx}},s_{\ref{lem:approx}}, \nu,\eps_{\ref{lem:approx}})$ and the last four parameters are as seen in~\eqref{eq:const}.
Each member $B \in \mathcal{B}_2$ has the property that $B[U_i,U_j]$ is $\eps_{\ref{lem:approx}}(t_2)$-regular whenever $i,j \in [t_2]$ are distinct. In addition, the members of $\mathcal{B}_2$ approximate the densities of the members of $\mathcal{B}'_1$ per~(APX.2).

Fix indices $1 \leq i < j \leq t_2$. In what follows, the members of $\mathcal{B}_2^{ij} = \left\{B[U_i,U_j]\colon B \in \mathcal{B}_2 \right\}$ are {\it sliced} so as to yield a collection of $\ell_2^2$ bipartite graphs such that each of them is $(\eps(t_2),\ell_2^{-2})$-regular. This is attained through {\it randomly slicing} each member of $\mathcal{B}_2^{ij}$ so that, with positive probability, each {\it slice} thus produced across all choices for $i$ and $j$ has the specified density and regularity. All this is carried out in two steps, namely Steps~I.a and I.b. In the first step, the so-called {\it dense} members of $\mathcal{B}^{ij}$ are sliced; in the second step the so-called {\it sparse} members of $\mathcal{B}^{ij}$ are sliced along with leftovers incurred through the slicing of the dense members.

\subsubsection*{Step I.a: Random slicing of dense parts} Fix indices $1 \leq i < j \leq t_2$ and let
\[
	\mathcal{D}^{ij} = \left\{B[U_i,U_j]\colon B \in \mathcal{B}_2\; \text{and} \; d(B[U_i,U_j]) \geq \ell_2^{-2} \right\}
\]
denote the members of~$\mathcal{B}_2^{ij}$ that are sufficiently dense; note that $\mathcal{D}^{ij} \neq \emptyset$ since the members of~$\mathcal{B}_2^{ij}$ partition $K^{(2)}(U_i,U_j)$ and $|\mathcal{B}_2^{ij}| = \ell_2$. For every $\Gamma \in \mathcal{D}^{ij}$, there exist an integer $0 \leq k_\Gamma \leq \ell_2^2$ and a real number $0 \leq \eta_\Gamma < \ell_2^{-2}$ such that $d(\Gamma) = \frac{k_\Gamma}{\ell_2^2} + \eta_\Gamma$. Colour the members of $E(\Gamma)$ by assigning each edge a colour from the palette $\{0,\ldots,k_\Gamma\}$, independently from the rest of the edges, according to the following scheme:
\begin{enumerate}
	\item An edge of $\Gamma$ is assigned the colour $0$ with
	      probability $\frac{\eta_\Gamma}{d(\Gamma)}$.

	\item An edge of $\Gamma$ is assigned the colour $1 \leq i
		      \leq k_\Gamma$ with probability $\frac{1}{\ell_2^2
			      d(\Gamma)}$.
\end{enumerate}

Note that
$$
	\frac{k_\Gamma}{\ell_2^2d(\Gamma)} + \frac{\eta_\Gamma}{d(\Gamma)} = \frac{1}{d(\Gamma)}\left(\frac{k_\Gamma}{\ell_2^2} + \eta_\Gamma \right)=1.
$$
Let $\Gamma_0,\Gamma_1,\ldots,\Gamma_{k_\Gamma}$ denote the random pairwise edge disjoint subgraphs of $\Gamma$ resulting from such a colouring. The random subgraph $\Gamma_0$ is referred to as the {\em trash subgraph}.

Set $C_2 = C_1/2$ and note that an appropriate choice of $C_1$ ensures that $C_2 \gg \ell_2^2$ holds. Fix a constant $\zeta$ satisfying
\begin{equation}\label{eq:zeta1}
	0 < \zeta \ll \min \left\{\ell_2^{-2}, \eps(t_2)/C_1 \right\}
\end{equation}
and
\begin{equation}\label{eq:zeta2}
	\left\lfloor \frac{1}{\ell_2^{-2} -\zeta} \right\rfloor = \ell_2^2 = \left\lceil \frac{1}{\ell_2^{-2} +\zeta} \right\rceil \quad \text{and} \quad \zeta + \max_{\Gamma \in \mathcal{D}^{ij}} \eta_\Gamma < \ell_2^{-2}.
\end{equation}
Consider the events
\begin{align*}
	\mathcal{E}_0(\Gamma) & = \left\{d(\Gamma_0) \leq \eta_\Gamma + \zeta \right\},                                                         \\
	\mathcal{E}_1(\Gamma) & = \left\{d(\Gamma_i) = \ell_2^{-2} \pm \zeta \; \text{for every}\; 1 \leq i \leq k_\Gamma \right\},             \\
	\mathcal{E}_2(\Gamma) & = \left\{\text{$\Gamma_i$ is $(\eps(t_2)/C_2,\ell_2^{-2})$-regular for every $1 \leq i \leq k_\Gamma$}\right\}.
\end{align*}
Finally, for every $k \in \{0,1,2\}$, let
$$
	\mathcal{E}_k = \{\text{$\mathcal{E}_k(\Gamma)$ holds for all $\Gamma \in \mathcal{D}^{ij}$ and for all indices $1 \leq i < j \leq t_2$}\}.
$$

We claim that
\begin{equation}\label{eq::E0E1E2}
	\Pr(\mathcal{E}_0 \wedge \mathcal{E}_1 \wedge \mathcal{E}_2)
	>
	0\,.
\end{equation}
Note that the number of pairs of indices $1 \leq i < j \leq t_2$ is independent of $n$. Moreover, $|\mathcal{D}^{ij}|$ is independent of $n$ for any given pair of indices $1 \leq i < j \leq t_2$.

Hence, in order to prove~\eqref{eq::E0E1E2}, it suffices to prove that
\begin{equation}\label{eq:prob-Gamma}
	\Pr\big(\mathcal{E}_0(\Gamma) \wedge \mathcal{E}_1(\Gamma) \wedge \mathcal{E}_2(\Gamma)\big)
	=
	1 - o_n(1)
\end{equation}
for every pair of indices $1 \leq i < j \leq t_2$ and every $\Gamma \in \mathcal{D}^{ij}$.

Note that $\mathcal{E}_1$ is not part of the formulation of Property~(R.2) but rather an auxiliary event facilitating our arguments seen in {Step~I.b} below. On the other hand, $\mathcal{E}_2$ is directly related to Property~(R.2) and is of prime concern in regaining Property~(R.2).

As noted above, to conclude { Step~I.a}, it remains to prove~\eqref{eq:prob-Gamma}. We begin by noting that
$\Ex[e(\Gamma_0)] = \eta_\Gamma |U_i||U_j|$ and that $\Ex[e(\Gamma_h)] = \frac{1}{\ell_2^2} |U_i||U_j|$ holds for every $1 \leq h \leq k_\Gamma$. Then, owing to Chernoff's inequality~\cite{JLR}*{Theorem~2.1 and Corollary~2.3}, we may write
\begin{align*}
	\Pr\big(e(\Gamma_0) \geq (\eta_\Gamma + \zeta) |U_i||U_j|\big)
	 & \leq
	\exp\left(-\frac{(\zeta|U_i||U_j|)^2}{2 (\eta_\Gamma +\zeta) |U_i||U_j|}\right)
	=
	\exp\big(-\Omega(n^2)\big)\,, \\
	\Pr\big(\left|e(\Gamma_h) - \ell_2^{-2} |U_i||U_j|\right| \geq \zeta|U_i||U_j|\big)
	 & \leq
	\exp(-\Omega(|U_i||U_j|))
	=
	\exp\big(-\Omega(n^2)\big)\,,
\end{align*}
where the last equalities seen for each of the bounds just specified, are owing to $\mathcal{V}_2$ being equitable, leading to $|U_i| = \Omega(n)$ holding for every $i \in [t_2]$. In particular, each of the events $\mathcal{E}_0(\Gamma)$ and $\mathcal{E}_1(\Gamma)$ holds asymptotically almost surely.

To estimate the probability that $\mathcal{E}_2(\Gamma)$ holds, fix $X \subseteq U_i$ and $Y \subseteq U_j$, and recall that
$$
	e(\Gamma[X,Y]) = d(\Gamma) |X||Y| \pm \eps_{\ref{lem:approx}}(t_2) |U_i||U_j|.
$$
For every $1 \leq h \leq k_\Gamma$, it holds that
\begin{align*}
	\Ex[e(\Gamma_h[X,Y])] & = \ell_2^{-2}|X||Y| \pm \frac{\eps_{\ref{lem:approx}}(t_2)}{\ell_2^2 d(\Gamma)}|U_i||U_j| \\
	                      & = \ell_2^{-2} |X||Y| \pm  \frac{\eps(t_2)}{C_1 \ell_2^2 d(\Gamma)} |U_i||U_j|             \\
	                      & = \ell_2^{-2} |X||Y| \pm \frac{\eps(t_2)}{C_1} |U_i||U_j|,
\end{align*}
where the last equality holds since $\Gamma \in \mathcal{D}^{ij}$ and thus $\ell_2^2 d(\Gamma) \geq 1$. It then follows by Chernoff's inequality~\cite{JLR}*{Theorem~2.1} that

\begin{align*}
	\Pr\Big(e(\Gamma_h[X,Y])  \geq  \ell_2^{-2} |X||Y| + \frac{\eps(t_2)}{C_1}|U_i||U_j| + \zeta |U_i||U_j| \Big)
	 & =
	\exp( -\Omega(n^2))
	\intertext{and}
	\Pr\Big(e(\Gamma_h[X,Y])  \leq  \ell_2^{-2} |X||Y| - \frac{\eps(t_2)}{C_1}|U_i||U_j| - \zeta |U_i||U_j| \Big)
	 & =
	\exp(-\Omega(n^2))
\end{align*}
both hold. The number of choices for the sets $X$ and $Y$ is at most $2^{2n}$ and the number of choices for $h$ is $k_\Gamma$. Since, moreover,  $C_2 = C_1/2 \gg \ell_2^2$ and $\zeta$ satisfies~\eqref{eq:zeta1}, it follows that  the random subgraph $\Gamma_h$ is a.a.s.\ $(\eps(t_2)/C_2, \ell_2^{-2})$-regular for every $1 \leq h \leq k_\Gamma$. This shows that~$\mathcal{E}_2(\Gamma)$ holds a.a.s.\ and thus concludes the proof of~\eqref{eq:prob-Gamma}.

\subsubsection*{Step~I.b: Randomly slicing the trash and sparse members of $\mathcal{B}_2^{ij}$} Fix indices $1 \leq i < j \leq t_2$. Expose $\Gamma_0,\Gamma_1,\ldots,\Gamma_{k_\Gamma}$ for every $\Gamma \in \mathcal{D}^{ij}$. Let $M = \bigcup_{\Gamma \in \mathcal{D}^{ij}} \bigcup_{i=1}^{k_\Gamma} \Gamma_i$ and $L = K^{(2)}(U_i,U_j) \setminus M$. Let $k = \sum_{\Gamma \in \mathcal{D}^{ij}} k_\Gamma$ so that one may write $M = \bigcup_{i=1}^k S_i$, where $S_i$ is $(\eps(t_2)/C_2,\ell_2^{-2})$-regular with density $d(S_i) = \ell_2^{-2} \pm \zeta$ for every $i \in [k]$. Note that $k \leq \ell_2^2$ holds by the first equality appearing in~\eqref{eq:zeta2}.

Since $S_1, \ldots, S_k$ are pairwise edge disjoint, $M$ is $\big(\frac{k\eps(t_2)}{C_2},\frac{k}{\ell_2^2}\big)$-regular.
The (bipartite) complement of $M$ in $K^{(2)}(U_i,U_j)$, namely $L$, is then $\big(\frac{k\eps(t_2)}{C_2},\frac{\ell_2^2-k}{\ell_2^2} \big)$-regular. Slice $L$ uniformly at random using $c = \ell_2^2 - k$ colours; let $L_1,\ldots, L_{c}$ denote the resulting slices (note that $c > 0$ as $L$ contains $\Gamma_0$ for every $\Gamma \in \mathcal{D}^{ij}$ and is thus dense). To gauge the regularity of $L_h$, where $h \in [c]$, note that for a fixed $X \subseteq U_i$ and $Y \subseteq U_j$, one has
$$
	\Ex[e_{L_h}(X,Y)] = \ell_2^{-2} |X||Y| \pm \frac{k\eps(t_2)}{C_2 \cdot c}|U_i||U_j| = \ell_2^{-2} |X||Y| \pm \frac{\eps(t_2)}{C_3}|U_i||U_j|
$$
where $C_3 = C_2 \cdot c/k$; an appropriate choice of $C_1$ ensures that $C_3 \gg \ell_2^2$ holds. An application of Chernoff's inequality (with deviation $\zeta'|U_i||U_j|$, where $0 < \zeta' \ll \eps(t_2)/C_3$), yields that a.a.s.~$L_h$ is $(\eps(t_2)/C_4, \ell_2^{-2})$-regular, where $C_4 \gg \ell_2^2$ is some constant. Since $c$ is independent of $n$, it follows that a.a.s.\ all of the aforementioned slices admit this level of regularity.

We conclude {Step~I} by setting $\mathcal{B}_3$ to denote the collection of all slices produced in {Steps~I.a} and~{I.b}.
As such, $\mathcal{B}_3$ partitions $K^{(2)}(\mathcal{V}_2)$ but need {\it not} be a refinement of $\mathcal{B}_2$ on account of {Step~I.b}. In addition, $\mathcal{B}_3$ is $\ell_2^2$-equitable (indeed, for every pair of indices $1 \leq i < j \in t_2$, some $k \leq \ell_2^2$ slices are created in {Steps~I.a} and $\ell_2^2 - k$ additional slices are created in~{I.b}) with each of its members being $(\eps(t_2),\ell_2^{-2})$-regular. The pair of partitions $(\mathcal{V}_2,\mathcal{B}_3)$ then satisfies Property~(R.2) with the aforementioned parameters.

\subsubsection*{Step~II: Regaining Property~(R.3)} In this step, we produce an equitable vertex partition $\mathcal{V}_3$ such that $\mathcal{V}_3 \prec \mathcal{V}_2$ and subsequently a partition $\mathcal{B}_4$ of $K^{(2)}(\mathcal{V}_3)$ such that the pair $(\mathcal{V}_3, \mathcal{B}_4)$ satisfies Properties~(R.1-3) with the correct parameters.

Let $\mathcal{V}_3$ be the equitable vertex partition resulting from an application of the Weak Lemma  (Lemma~\ref{lem:weak}) with $\mathcal{V}_2$ as the initial vertex partition and along with
$$
	\delta_{\ref{lem:weak}} = \delta_2\left(\ell_2^2\right)\quad \text{and} \quad t_{\ref{lem:weak}} = t_2.
$$

Recall that $\mathcal{V}_2$ has the form $V = U_1 \discup \cdots \discup U_{t_2}$ and  that $\mathcal{V}_3 \prec \mathcal{V}_2$. Since both $\mathcal{V}_2$ and $\mathcal{V}_3$ are equitable, it follows that the number of members of $\mathcal{V}_3$ refining a single cluster of $\mathcal{V}_2$ is uniform across all clusters of $\mathcal{V}_2$. For a pair of distinct indices $i,j \in [t_2]$, let
$$
	\mathcal{W}^{(i)} = \big\{ W^{(i)}_1,\ldots,W^{(i)}_z \big\}\quad \text{and}\quad 
	\mathcal{W}^{(j)} = \big\{W^{(j)}_1,\ldots,W^{(j)}_z \big\}
$$
denote the members of $\mathcal{V}_3$ refining $U_i$ and $U_j$, respectively. In addition, let $\mathcal{B}_3^{ij}$ consist of the members of $\mathcal{B}_3$ partitioning $K^{(2)}(U_i,U_j)$; recall that $|\mathcal{B}_3^{ij}| = \ell_2^2$.

Fix some $W \in \mathcal{W}^{(i)}$, $W' \in \mathcal{W}^{(j)}$, and $\Gamma \in \mathcal{B}^{ij}_3$. We claim that $\Gamma[W,W']$ is $(\delta_2(\ell_2^2),\ell_2^{-2})$-regular. Indeed, note first that $|W|, |W'| \geq \frac{n}{T_{\ref{lem:weak}}(t_2,\delta_2\left(\ell_2^2\right))}$ and that $T_{\ref{lem:weak}}(t_2,\delta_2\left(\ell_2^2\right))^{-1} \geq \eps(t_2)$, where the latter inequality holds by~\eqref{eq:eps-func}. Hence, an application of the Slicing Lemma (Lemma~\ref{lem:slicing}) with

$$
	d_{\ref{lem:slicing}} = \ell_2^{-2}, \quad \delta_{\ref{lem:slicing}} = \eps(t_2), \quad \alpha_{\ref{lem:slicing}} = T_{\ref{lem:weak}}(t_2,\delta_2\left(\ell_2^2\right))^{-1}
$$
implies that $\Gamma[W,W']$ is $\left(\xi,\ell_2^{-2}\pm \delta_{\ref{lem:slicing}}\right)$-regular, with
$$
	\xi = \max \left\{T_{\ref{lem:weak}}(t_2,\delta_2\left(\ell_2^2\right)) \delta_{\ref{lem:slicing}}, 2\delta_{\ref{lem:slicing}}\right\} \leq \frac{\delta_2(\ell_2^2)}{T_{\ref{lem:weak}}(t_2,\delta_2\left(\ell_2^2\right))},
$$
where the above inequality holds by~\eqref{eq:eps-func}. We may then absorb the deviation in the density by enlarging the error term to deduce that $\Gamma[W,W']$ is $(\delta_2(\ell_2^2),\ell_2^{-2})$-regular, as claimed.

We conclude that the members of $\mathcal{B}_3^{ij}$ define $\ell_2^2$ edge disjoint $(\delta_2(\ell_2^2),\ell_2^{-2})$-regular subgraphs, between every pair of sets $W \in \mathcal{W}^{(i)}$ and $W' \in \mathcal{W}^{(j)}$. Moreover, these subgraphs partition $K^{(2)}(W,W')$. Define $\mathcal{B}_4$ to be the partition of $K^{(2)}(\mathcal{V}_3)$ whose members are the subgraphs of the form $\Gamma[W,W']$, where $W \in \mathcal{W}^{(i)}$, $W' \in \mathcal{W}^{(j)}$, and $1 \leq i < j \leq t_2$; note that $\mathcal{B}_4 \prec \mathcal{B}_3$.

\subsubsection*{Index increment} We conclude the proof of Proposition~\ref{lem:ghrl} by tracking the index of the various partitions defined throughout the refinement process above and prove that this process culminates with the last partition, namely $\mathcal{B}_4$, satisfying
\begin{equation}\label{eq:final-index-inc}
	\mathrm{ind}(\mathcal{B}_4) \geq \mathrm{ind}(\mathcal{B}_1) + \delta_3^4/8.
\end{equation}

The refinement process commences with the Index Increment Lemma (Lemma~\ref{lem:index-inc}) yielding $\mathrm{ind}(\mathcal{B}'_1) \geq \mathrm{ind}(\mathcal{B}_1) + \delta_3^4/2$, where here the index is taken with respect to $\mathcal{V}_1$ and the partition~$\mathcal{H}$ whose members are~$H$ and its complement. The pair of partitions $(\mathcal{V}_2,\mathcal{B}_2)$ is obtained through an application of the Approximation Lemma (Lemma~\ref{lem:approx}) to the pair $(\mathcal{V}_1,\mathcal{B}'_1)$. It follows that
$$
	\mathrm{ind}(\mathcal{B}_2) \overset{\eqref{eq:ind-approx-new}}{\geq} \mathrm{ind}(\mathcal{B}'_1) - 9 (2s)^3 \nu \overset{\eqref{eq:const}}{\geq} \mathrm{ind}(\mathcal{B}_1) + \delta_3^4/4
$$
holds.

The partition $\mathcal{B}_3$ need not be a refinement of $\mathcal{B}_2$; this is due to the treatment of sparse and trash subgraphs of $\mathcal{B}_2$ seen in {Step~I.b}, where these subgraphs are united and then collectively sliced. The partition $\mathcal{B}_2$ is obtained from partition $\mathcal{B}'_1$ through the approximation lemma (Lemma~\ref{lem:approx}) and thus $|\mathcal{B}_2| = |\mathcal{B}'_1| = \ell_2$ holds. Let $\mathcal{A}_3 = \{B_3 \in \mathcal{B}_3 \colon B_3 \not\subseteq B_2 \text{ for every } B_2 \in \mathcal{B}_2 \}$. Then,
\begin{align*}
	\sum_{A \in \mathcal{A}_3} e(A) & \overset{\phantom{\eqref{eq:zeta2}}}{\leq} \sum_{1 \leq i < j \leq t_2} \left(\sum_{\Gamma \in \mathcal{D}^{ij}} e(\Gamma_0) + \sum_{\Gamma \in \mathcal{B}_2^{ij} \setminus \mathcal{D}^{ij}} e(\Gamma)\right) \\
	                                & \overset{\phantom{\eqref{eq:zeta2}}}{\leq} |\mathcal{B}_2| \left(\zeta+\max_{\Gamma \in \mathcal{D}^{ij}}\eta_\Gamma\right) n^2 + |\mathcal{B}_2| \cdot\ell_2^{-2} n^2                                          \\
	                                & \overset{\eqref{eq:zeta2}}{\leq} 2\ell_2^{-1} n^2.
\end{align*}
It follows that $\mathcal{B}_3$ is a $(2\ell_2^{-1})$-refinement of $\mathcal{B}_2$ and thus
$$
	\mathrm{ind}(\mathcal{B}_3) \geq \mathrm{ind}(\mathcal{B}_1) + \delta_3^4/4 - 2\ell_2^{-1} \overset{\eqref{eq:ell_2}}{\geq} \mathrm{ind}(\mathcal{B}_1) + \delta_3^4/8
$$
holds, by the index approximation lemma (Lemma~\ref{lem:refine}).

The last partition, namely $\mathcal{B}_4$, is attained from $\mathcal{B}_3$ through repeated applications of Lemma~\ref{lem:slicing}. As such $\mathcal{B}_4 \prec_0 \mathcal{B}_3$ and thus $\mathrm{ind}(\mathcal{B}_4) \geq \mathrm{ind}(\mathcal{B}_3)$ holds by the index approximation lemma (Lemma~\ref{lem:refine}). This proves~\eqref{eq:final-index-inc} as required.

\subsubsection*{Conclusion} The fact that the index of a partition is bounded from above by one (see~\eqref{eq:index-bound}) coupled with the index increment obtained in each iteration of the refinement process, lead to a pair of partitions satisfying Properties~(R.1-4) being encountered within $O(\delta_3^{-4})$ iterations of this process. This concludes the proof of Proposition~\ref{lem:ghrl}.
\end{proof}

\section{Concluding remarks}\label{sec:concluding}

\subsection{Alternatives to the tuple lemma}

In this section, we provide two alternatives to the tuple property (Proposition~\ref{lem:tuple}) that avoid using the Strong Lemma (Lemma~\ref{lem:shrl}), yet fall shy from being a suitable replacement for the Tuple Lemma in our proof of Theorem~\ref{thm:main-Ramsey}. We start with the following simple adaptation of~\cite{KST}*{Lemma~3.4}. Given a hypergraph $H$, let
$$
	\delta_1(H) = \min\{\deg_H(v) \colon v \in V(H)\}
$$
denote the minimum $1$-degree of $H$.

\begin{lemma}\label{lem:weak-tuple-lemma}
	For every $d >0$, positive integer $t$, and $c_1 \in (0,1]$, there exist $c_2,c_3 \in (0,1]$ and $n_0 \in \mathds{N}$ such that the following holds whenever $n \geq n_0$. Let $H$ be an $n$-vertex hypergraph satisfying $\delta_1(H) \geq dn^2$ and let $U \subseteq V(H)$ be a set of size $|U| \geq c_1 n$. Then, there exist at least~$c_2 n^t$ members of $U^{(t)}$ whose joint link graph has size at least $c_3 n^2$.
\end{lemma}

\begin{proof}
	Fix $U \subseteq V(H)$ of size $|U| \geq c_1 n$. Note that
	\begin{equation}\label{eq:deg-sum-U}
		\sum_{\{w,v\} \in V(H)^{(2)}} \deg_H(w,v,U) = \sum_{u \in U} \deg_H(u) \geq |U|\cdot  \delta_1(H) \geq d \cdot c_1 n^3,
	\end{equation}
	where
	$$
		\deg_H(w,v,U) = \left|\{u \in U\colon \{u,v,w\} \in E(H) \} \right|.
	$$
	A double counting argument and an application of Jensen's inequality imply that
	\begin{align}
		\sum_{X \in U^{(t)}} e(L_H(X)) & \overset{\phantom{\eqref{eq:deg-sum-U}}}{=} \sum_{\{w,v\} \in V(H)^{(2)}} \binom{\deg_H(w,v,U)}{t} \nonumber               \\
		                                    & \overset{\phantom{\eqref{eq:deg-sum-U}}}{\geq} \binom{n}{2} \cdot \binom{\binom{n}{2}^{-1} \sum_{w,v} \deg_H(w,v,U}{t}\nonumber \\
		                                    & \overset{\eqref{eq:deg-sum-U}}{\geq} \binom{n}{2} \cdot \binom{\binom{n}{2}^{-1} \cdot dc_1 n^3}{t} \nonumber                   \\
		                                    & \overset{\phantom{\eqref{eq:deg-sum-U}}}{=} \binom{n}{2} \cdot \binom{c_4 n}{t} \nonumber                                       \\
		                                    & \overset{\phantom{\eqref{eq:deg-sum-U}}}{\geq} c_5 n^{t+2}, \label{eq:link-graph-sum}
	\end{align}
	where $c_4$ and $c_5$ are appropriate positive constants. Trivially, $e(L_H(X)) \leq n^2$ holds for every $X \in U^{(t)}$. The existence of the constants $c_2$ and $c_3$ thus follows by~\eqref{eq:link-graph-sum}, concluding the proof of the lemma.
\end{proof}

It is evident that Lemma~\ref{lem:weak-tuple-lemma} offers much weaker control over the joint link graphs of $t$-tuples than that which is ensured by our new tuple property (Proposition~\ref{lem:tuple}); in particular, it is insufficient for our proof of Theorem~\ref{thm:main-Ramsey}. Indeed, an important step of our proof involves finding a blue (say) copy $K \subseteq \Gamma[X]$ of $\tilde{K}_{t/2}$ such that its vertex  set forms a good tuple, that is, the joint link graph of this tuple is sufficiently dense. Alas, the parameters of Lemma~\ref{lem:weak-tuple-lemma} do not guarantee the existence of such a copy.

A versatile tool, commonly used to replace the Tuple Lemma for graphs (Lemma~\ref{lem:tuple-graphs}), and consequently avoid the use of the (graph) regularity lemma altogether, is the so-called {\em dependent random choice}~\cite{FS11}. A variant of this tool for hypergraphs was considered before in~\cite{CFS09} by Conlon, Fox and Sudakov; their version, however, is not strictly aligned with the tuple property we seek.  A formulation of the dependent random choice fitting the settings encountered in our proof of Theorem~\ref{thm:main-Ramsey} reads as follows.

\begin{proposition}\label{prop:DRC}{\em (Dependent random choice for hypergraphs)}
	Let $a,m,n,r$ be positive integers and let $H$ be an $n$-vertex hypergraph. If there exists a positive integer $t$ such that
	\begin{equation}\label{eq:DRC}
		n^{1-2t} \delta_1(H)^t - \binom{n}{r}\left(\frac{2m}{n(n-1)} \right)^t \geq a
	\end{equation}
	holds, then there exists a subset $U \subseteq V(H)$ of size $|U| \geq a$ such that $e(L_H(S)) \geq m$ for every $S \in U^{(r)}$.
\end{proposition}

\begin{proof}
	Fix a positive integer $t$ satisfying~\eqref{eq:DRC} and let $\left\{ \boldsymbol{b}_1 = \{u_1,v_1\}, \ldots, \boldsymbol{b}_t = \{u_t,v_t\} \right\}$ be $t$ pairs of vertices each chosen uniformly at random with replacement from $V(H)^{(2)}$, independently from one another. For every $i \in [t]$ let $N_H(\boldsymbol{b}_i) = \{w \in V(H)\colon \{u_i,v_i,w\} \in E(H)\}$ and let
	$$
		X = |N_H(\boldsymbol{b}_1,\ldots,\boldsymbol{b}_t)| = \bigg|\bigcap_{i=1}^t N_H(\boldsymbol{b}_i)\bigg|.
	$$
	Then,
	$$
		\Ex[X] = \sum_{v \in V(H)}\Pr\big(v \in N_H(\boldsymbol{b}_i)\; \text{for every $i \in [t]$}\big)
		\geq \sum_{v \in V(H)} \left(\frac{\deg_H(v)}{n^2}\right)^t
		\geq n^{1-2t}\delta_1(H)^t.
	$$
	Any subset of vertices $S \subseteq V(H)$ satisfies $S \subseteq N_H(\boldsymbol{b}_1,\ldots,\boldsymbol{b}_t)$ with probability at most $\left(\frac{2 e(L_H(S))}{n(n-1)}\right)^t$.
	Consequently, $\Ex[Y]  \leq \binom{n}{r}\left(\frac{2m}{n(n-1)}\right)^t$ holds, where
	$$
		Y= \left|\left\{S \in N_H(\boldsymbol{b}_1,\ldots,\boldsymbol{b}_t)^{(r)}\colon e(L_H(S)) < m \right\} \right|.
	$$
	Therefore
	$$
		\Ex[X-Y] \geq n^{1-2t} \delta_1(H)^t - \binom{n}{r}\left(\frac{2m}{n(n-1)} \right)^t \geq a, $$
	where the last inequality is owing to the assumption that $t$ satisfies~\eqref{eq:DRC}. It follows that the required set $U$ exists, thus concluding the proof of the proposition.
\end{proof}

The assertion of Proposition~\ref{prop:DRC} includes a restriction imposed on $\delta_1(H)$. This restriction allows us to prove this proposition while avoiding the traditional appeal to Jensen's inequality, seen in the proof of the original formulation fitting the graph setting~\cite{FS11}.

The following consequence of the dependent random choice for hypergraphs, namely Proposition~\ref{prop:DRC}, is more inline with the formulation of the tuple property (Proposition~\ref{lem:tuple}).

\begin{corollary}\label{cor:DRC-dense}
	For every $\alpha,\rho >0$, and positive integer $r$, there exist $n_0 \in \mathds{N}$ and $\beta >0$ such that the following holds whenever $n \geq n_0$. Let $H$ be an $n$-vertex hypergraph satisfying $\delta_1(H) \geq \alpha n^2$. Then, there exists a  subset $U \subseteq V(H)$ of size $|U| \geq \beta n$ such that every member $S \in U^{(r)}$ satisfies $e(L_H(S)) \geq n^{2-\rho}$.
\end{corollary}

\begin{proof}
	Set  $m = n^{2-\rho}$ and $t = \lceil r/\rho \rceil$. Then,

	$$
		n^{1-2t} \delta_1(H)^t - \binom{n}{r}\left(\frac{2m}{n(n-1)} \right)^t \geq
		n^{1-2t} \alpha^t n^{2t} - \binom{n}{r}\left(\frac{3 n^{2-\rho}}{n^2} \right)^t \geq \alpha^t n - O(n^{r-\rho t}) \geq \beta n,
	$$
	where $\beta > 0$ is an appropriate constant. The assertion of Corollary~\ref{cor:DRC-dense} thus follows by Proposition~\ref{prop:DRC}.
\end{proof}

We have attempted to replace the tuple property (Proposition~\ref{lem:tuple}) with Corollary~\ref{cor:DRC-dense} in the proof of Theorem~\ref{thm:main-Ramsey}. This, however, has resulted in a higher edge probability for the random perturbation. The main cause for this problem is an adaptation of~\cite{DT19}*{Theorem~2.10(iii)} mandating that the random perturbation $H \sim \mathds{H}^{(3)}(n,p)$ a.a.s.\ satisfy the property by which there exist constants $\beta = \beta(t) > 0$ and $C = C(t) > 0$
such that $H$ is $(\tilde K^{(3)}_t, \tilde K^{(3)}_{t/2}, n^{- \beta'})$-Ramsey, whenever $0 < \beta' \leq \beta$ and $p = p(n) \geq C n^{-(1-\beta')/M_{t,t/2}}$.

\subsection{Further research}\label{sec:further-research}

As mentioned in the introduction, we conjecture that Theorem~\ref{thm:main-Ramsey} uncovers the threshold for the emergence of monochromatic expanded cliques in randomly perturbed dense hypergraphs and can be complemented as follows.

\begin{conjecture}\label{con:main-Ramsey}
	For every even integer $t \geq 6$ there exist constants $d$, $c>0$, and there exists a sequence of
	$3$-uniform $n$-vertex hypergraphs $(H_n)_{n\in\NN}$ with $e(H_n)\geq dn^3$ for every $n\in\NN$ such that
	\[
		\lim_{n \to \infty} \Pr\big(H_n \cup \mathds{H}^{(3)}(n,p) \to (\tilde K^{(3)}_t)\big)
		=
		0\,,
	\]
	whenever $p \leq cn^{-1/M}$ for $M=m_3\big(\tilde K^{(3)}_t,  \tilde K^{(3)}_{t/2}\big)$.
\end{conjecture}

Conjecture~\ref{con:main-Ramsey} may hold for $t=4$ as well. However, this value is excluded due to the distinct behaviour seen in the graph case~\cites{DT19,Powierski19}. The proof of Theorem~\ref{thm:main-Ramsey} presented here extends for the 
asymmetric Ramsey property 
$H\lra(\tilde K^{(3)}_t,\tilde K^{(3)}_s)$ for sufficiently large integers $t\geq s$ and $M$ replaced by 
$m_3\big(\tilde K^{(3)}_t,\tilde K^{(3)}_{\lceil s/2 \rceil}\big)$. It seems plausible that 
the corresponding generalisation of Conjecture~\ref{con:main-Ramsey} may also hold.

\appendix

\section{Expanded cliques in randomly perturbed dense hypergraphs}\label{app:Turan}
In this section, we determine the threshold for the property
$\tilde K^{(3)}_t \subseteq H \cup \mathds{H}^{(3)}(n,p)$, whenever~$H$ is a dense hypergraph
and show that it is smaller than $n^{-1/m(\tilde K^{(3)}_t)}$.

\begin{proposition}\label{prop:Turan}
	For every every integer $t \geq 4$ the following holds.
	\begin{enumerate}[label=\alabel]
		\item\label{it:1} For every $d>0$ and every sequence of $3$-uniform $n$-vertex hypergraphs $(H_n)_{n\in\NN}$ with $e(H_n)\geq dn^3$ for every $n\in\NN$ we have
		      \[
		      	\lim_{n \to \infty} \Pr\big(\tilde K^{(3)}_t \subseteq H \cup \mathds{H}^{(3)}(n,p)\big) = 1\,,
		      \]
		      whenever $p \gg n^{-\frac{1}{m(\lceil t/3 \rceil)}}$.
		\item\label{it:0} There exists a constant $d>0$ and there exists a sequence of $3$-uniform $n$-vertex hypergraphs $(H_n)_{n\in\NN}$
		      with $e(H_n)\geq dn^3$ for every $n\in\NN$  such that
		      \[
		      	\lim_{n \to \infty} \Pr\big(\tilde K^{(3)}_t \subseteq H_n \cup \mathds{H}^{(3)}(n,p)\big) = 0\,,
				\]
		      whenever $p \ll n^{-\frac{1}{m(\lceil t/3 \rceil)}}$.
	\end{enumerate}
\end{proposition}

A key tool in our proof of Proposition~\ref{prop:Turan} is a well-known result by Janson, regarding random variables of the form $X = \sum_{A \in \mathcal{S}}I_A$. Here, $\mathcal{S}$ is a family of non-empty subsets of some ground set $\Omega$ and $I_A$ is the indicator random variable for the event $A \subseteq \Omega_p$, where $\Omega_p$ is the so-called {\it binomial random set} arising from including every element of $\Omega$ independently at random with probability $p$. For such random variables, set $\lambda = \Ex[X]$,  and define
$$
	\Delta
	=  \frac{1}{2} \sum_{\stackrel{A,B \in \mathcal{S}}{A \neq B\text{ and } A \cap B \neq \emptyset}} \Ex[I_A I_B]\,.
$$
Janson's inequality (see, e.g.,~\cite{JLR}*{Theorem~2.18}) upper bounds the probability of nonexistence.

\begin{theorem}\label{thm:Janson-exist}
	For $X$, $\lambda$, and $\Delta$ as above we have
	$\mathds{P}(X = 0) \leq \exp \big(- \frac{\lambda^2}{\lambda + 2\Delta}\big)$.\qed
\end{theorem}

\begin{proof}[Proof of Proposition~\ref{prop:Turan}]
	Starting with the $0$-statement in part~\ref{it:0} of Proposition~\ref{prop:Turan}. Let $t \geq 4$ be an integer and for every integer $n$, let~$H_n$ be the complete tripartite $n$-vertex hypergraph with vertex  set $V_1 \discup V_2 \discup V_3$, where $|V_1| \leq |V_2| \leq  |V_3| \leq |V_1|+1$. Clearly $H_n$ satisfies the density assumption for $d=1/27$. 
	If $H_n \cup \mathds{H}^{(3)}(n,p)$ contains a copy $K$ of $\tilde{K}_t$, then one of the parts of $H_n$, say $V_1$, must contain at least $\lceil t/3 \rceil \geq 2$ of the branch vertices of~$K$. Since $H_n$ is tripartite, any edge which connects any two branch vertices of~$K$ which are both in $V_1$, must be in $\mathds{H}^{(3)}(n,p)$. It follows that $\mathds{H}^{(3)}(n,p)$ contains a copy of $\tilde{K}_{\lceil t/3 \rceil}$. However, the probability of the latter event is at most
	$$
		\binom{n}{v(\tilde{K}_{\lceil t/3 \rceil})} p^{e(\tilde{K}_{\lceil t/3 \rceil})} = o(1).
	$$

	Next, we prove the $1$-statement from part~\ref{it:1} of Proposition~\ref{prop:Turan}. Let $r = \lceil t/3 \rceil$, let $H=H_n$ be a sufficiently large $n$-vertex dense hypergraph, and let $R \sim \mathds{H}^{(3)}(n,p)$. Set $s = t^2$ and let~$T^s$ denote the complete tripartite hypergraph with each of its partition classes having size~$s$.
	\begin{claim}\label{cl::KtinTs}
		Let $T$ be a copy of $T^s$ with veretex-set $V(T) = V_1 \discup V_2 \discup V_3$. For every $i \in [3]$, let~$K_i$ be a copy of $\tilde{K}_r$ satisfying $V(K_i) \subseteq V_i$. Then, $T \cup K_1 \cup K_2 \cup K_3$ contains a copy of~$\tilde{K}_{3r}$.
	\end{claim}

	\begin{proof}
		For $i \in [3]$, let $v^i_1, \ldots, v^i_r$ denote the branch vertices of $V(K_i)$ and let $A_i = V_i \setminus V(K_i)$. 
		For every $1 \leq i < j \leq 3$, let $k_{ij}$ denote the unique element of $[3] \setminus \{i,j\}$. For every $1 \leq i < j \leq 3$, let $\psi_{ij} \colon \{v^i_1, \ldots, v^i_r\} \times \{v^j_1, \ldots, v^j_r\} \to A_{k_{ij}}$ be an injection; such a function exists since $|A_{k_{ij}}| = s - (r + \binom{r}{2}) \geq r^2$. For every $1 \leq i < j \leq 3$, let
		$$
			E_{ij} = \big\{\{v^i_a, v^j_b, \psi_{ij}(v^i_a, v^j_b)\} \colon (a,b) \in [r] \times [r]\big\}.
		$$
		Then, the sub-hypergraph of $T \cup K_1 \cup K_2 \cup K_3$, with edge set 
		\[
		E(K_1) \cup E(K_2) \cup E(K_3) \cup E_{12} \cup E_{13} \cup E_{23}\,,
		\]
		forms a copy of $\tilde{K}_{3r}$.\phantom\qedhere\hfill{$\blacksquare$}
	\end{proof}

	\begin{claim} \label{cl::3disKt}
		Let $\gamma > 0$ be a constant, let $N \geq \gamma n^{3s}$ be an integer, and let $T_1, \ldots, T_N$ be pairwise distinct copies of $T^s$ in $H$. For every $i \in [N]$, let $V^i_1 \discup V^i_2 \discup V^i_3$ denote the vertex  set of $T_i$. Then, a.a.s.\ there exists an $i \in [n]$ such that $\tilde{K}_r \subseteq R[V^i_j]$ for every $j \in [3]$.
	\end{claim}

	\begin{proof}
		Let $K_n^3$ denote the complete 3-graph with vertex  set $[n]$. Let $F$ be the 3-graph consisting of three pairwise vertex  disjoint copies of $\tilde{K}_r$; note that $m(F) = m(\tilde{K}_r)$. Let $\mathcal{F}$ be the family of labelled copies $F'$ of $F$ in $K_n^3$ for which there exists an index $i \in [N]$ such that $\tilde{K}_r \subseteq F'[V^i_j]$ for every $j \in [3]$; in this case $F'$ is said to be \emph{included} in $T_i$. Note that, for every $i \in [N]$, at least one element of $\mathcal{F}$ is included in $T_i$. Moreover, every member of $\mathcal{F}$ is included in at most $\binom{n}{s - v(\tilde{K}_r)}^3 = O(n^{3s - 3v(\tilde{K}_r)})$ of the $T_i$'s. We conclude that $|\mathcal{F}| = \Theta(n^{3v(\tilde{K}_r)}) = \Theta(n^{v(F)})$. For every $F' \in \mathcal{F}$, let $Z_{F'}$ be the indicator random variable for the event $F' \subseteq \mathds{H}^{(3)}(n,p)$. For the random variable $X_F = \sum_{F' \in \mathcal{F}} Z_{F'}$ 
		we aim to prove that $\mathds{P}(X_F = 0) = o(1)$. Note that
		$$
			\lambda = \mathds{E}(X_F) = \sum_{F' \in \mathcal{F}} p^{e(F')} = |\mathcal{F}| p^{e(F)} = \Theta(n^{v(F)} p^{e(F)}) = \omega(1).
		$$
		Writing $F_i \sim F_j$ whenever $(F_i, F_j) \in \mathcal{F} \times \mathcal{F}$ are distinct and not edge disjoint, we obtain
		\begin{align} \label{eq::Delta}
			\Delta(F)
			 &=
			\sum_{\substack{(F_i, F_j) \in \mathcal{F} \times \mathcal{F}      \\ F_i \sim F_j}} \Ex[Z_{F_i} Z_{F_j}]
			=
			\sum_{\substack{(F_i, F_j) \in \mathcal{F} \times \mathcal{F}      \\ F_i \sim F_j}} p^{e(F_i) + e(F_j) - e(F_i \cap F_j)}
			\nonumber                                                          \\
			 &=
			\sum_{\substack{J \subsetneq F                                     \\ e(J) \geq 1}}
			\sum_{\substack{(F_i, F_j) \in \mathcal{F} \times \mathcal{F}      \\ F_i \cap F_j \cong J}} p^{2e(F) - e(J)}
			=
			O_F \Bigg(n^{2v(F)} p^{2e(F)} \cdot \sum_{\substack{J \subsetneq F \\ e(J) \geq 1}}
			n^{-v(J)} p^{-e(J)} \Bigg)
			  =
			o(\lambda^2)\,,
		\end{align}
		where the last equality follows since $n^{v(J)} p^{e(J)} = \omega(1)$ holds for every $J \subsetneq F$ with at least one edge, by the choice of $p$ and the definition of $m(F) = m(\tilde{K}_r)$.

		It then follows by Theorem~\ref{thm:Janson-exist} that
		$$
			\mathds{P}(X_F = 0) \leq \exp \left(- \frac{\lambda^2}{\lambda + \Delta(F)}\right) = o(1)\,,
		$$
		which concludes the proof of Claim~\ref{cl::3disKt}\phantom\qedhere\hfill{$\blacksquare$}
	\end{proof}

	Returning to the main thread of the proof of the 1-statement of Proposition~\ref{prop:Turan}, note that~$H$ is dense and the Tur\'an density of $T^s$ is zero (see, e.g.,~\cites{Erdos64,Keevash11}). Hence, $H$ admits~$\Omega(n^{3s})$ distinct copies of $T^s$. Therefore, it follows by Claim~\ref{cl::3disKt} that a.a.s.\ there exists a copy $T$ of~$T^s$ in $H$, with vertex  set $V_1 \discup V_2 \discup V_3$, such that $R[V_i]$ contains a copy $K_i$ of $\tilde{K}_r$ for every $i \in [3]$. It then follows by Claim~\ref{cl::KtinTs} that $\tilde{K}_t \subseteq \tilde{K}_{3r} \subseteq T \cup K_1 \cup K_2 \cup K_3$.
\end{proof}

\section{Direct proof of the tuple lemma for joint links}
\label{sec:direct-tuple}
In this section, we present a direct proof of Proposition~\ref{lem:tuple}, that avoids the use of the counting lemma of Nagle and R\"odl~\cite{NR03}*{Theorem~9.0.2}.
Given $t, \eps,$ and $d_3$ per the premise of Proposition~\ref{lem:tuple}, set an auxiliary constant
\begin{equation}\label{eq:zeta}
	0 < \zeta \ll \min\left\{\eps,d_3^{t-1}, 1\right\}
\end{equation}
and choose
\begin{equation}\label{eq:delta3}
	0 < \delta_3 \ll \min \left\{(d_3^{t-1} - 11 \zeta)\zeta^2,\zeta^3\right\}.
\end{equation}
Given $0 < d_2 \ll d_3$, set
\begin{equation}\label{eq:delta2-r}
	0 < \delta_2 \ll \min \left\{\frac{\zeta^2 d_2^{6t}}{\max \{t, 1\}},\zeta^2(d_3^{t-1} -\zeta)^2d_2^{8t+2} \right\} \quad \text{and} \quad r = \frac{\zeta}{d_2^{2t-2}}.
\end{equation}
The assumption that $d_2 \ll d_3$ appearing in the premise supports a choice of $\zeta$ which ensures that $r$ is a positive integer.

The proof is by induction on $t$. For $t=0$, the set defined through the intersection seen on the left hand side of~\eqref{eq:tuple} coincides with $E_P(Y,Z)$; indeed, in this case, no members of~$X$ are taken into the tuple and thus no restriction is imposed through the aforementioned intersection. In this case,
$$
	e_P(Y,Z) = d_2 |Y||Z| \pm \eps d_2 |Y||Z|
$$
is asserted by~\eqref{eq:tuple} and this is supported by the assumption that $P$ forms a $(\delta_2,d_2)$-triad and $\delta_2 \leq \eps d_2$, holding by~\eqref{eq:delta2-r} and~\eqref{eq:zeta}.
Assume then that $t \geq 1$ and proceed to the induction step.
Throughout this section, we employ the following notation for joint vertex  neighbourhoods and joint link graphs.

\begin{definition}
	Let $k$ be a positive integer. Given $\boldsymbol{x} \in X^k$, let
	$$
		N_P(\boldsymbol{x},Y) = \left(\bigcap_{x \in \boldsymbol{x}} N_P(x)\right) \cap Y \quad \text{and} \quad N_P(\boldsymbol{x},Z) = \left(\bigcap_{x \in \boldsymbol{x}} N_P(x)\right) \cap Z
	$$
	denote the {\em joint vertex  neighbourhoods} of the members of $\boldsymbol{x}$ in $Y$ and in $Z$, respectively. In addition, let
	$$
		L_H(\boldsymbol{x},P) = \bigcap_{x \in \boldsymbol{x}} L_H(x,P)
	$$
	denote the {\em joint link graph of all members of $\boldsymbol{x}$ supported on $P$}.
\end{definition}

The core property pursued throughout the proof of Proposition~\ref{lem:tuple} reads as follows.

\begin{definition}[Extension properties for tuples]
	Let $t$ be a positive integer. A $(t-1)$-tuple $\boldsymbol{x} \in X^{t-1}$ is said to have the {\em lower}-$\gamma$-{\em extension property} if it satisfies the following properties simultaneously.
	\begin{description}
		\item [(E.1)]
		      $$
			      \left|N_P(\boldsymbol{x},Y) \right| = (1\pm \gamma)d_2^{t-1} |Y| \quad \text{and} \quad
			      \left|N_P(\boldsymbol{x},Z) \right| = (1\pm \gamma)d_2^{t-1} |Z|\,,
		      $$
		\item [(E.2)]
		      $$
			      e(L_H(\boldsymbol{x},P)) =
			      \left(d_3^{t-1} \pm \gamma\right)d_2^{2t-1}|Y||Z|\,, 
		      $$
		\item [(E.3)] and all but at most $\gamma|X|$ vertices
		      $x\in X$ satisfy
		      \begin{equation}\label{eq:E3}
			      e(L_H(\boldsymbol{x'},P)) \geq (d_3^t - 13
			      \gamma)d_2^{2t+1}|Y||Z|,
		      \end{equation}
		      where $\boldsymbol{x'} = (\boldsymbol{x},x)$ is a $t$-tuple.
	\end{description}
	We say that $\boldsymbol{x}$ satisfies the {\em upper}-$\gamma$-{\em extension property} if $\boldsymbol{x}$ satisfies Properties~(E.1-2) simultaneously along with
	\begin{description}
		\item [(E.4)] all but at most $\gamma|X|$ vertices
		      $x\in X$ satisfy
		      \begin{equation}\label{eq:E4}
			      e(L_H(\boldsymbol{x'},P)) \leq (d_3^t + 13
			      \gamma)d_2^{2t+1}|Y||Z|,
		      \end{equation}
		      where $\boldsymbol{x'} = (\boldsymbol{x},x)$ is a $t$-tuple.
	\end{description}
\end{definition}

To establish the induction step, it suffices to prove that there exists a constant $\gamma \in (0,\eps/13)$
such that the following two statements hold:
\begin{description}
	\item [(L)] all but at most $\frac{\eps}{2} |X|^{t-1}$ of the
	      members of $X^{t-1}$ satisfy the lower-$\gamma$-extension
	      property.
	\item [(U)] all but at most $\frac{\eps}{2} |X|^{t-1}$ of the members of
	      $X^{t-1}$ satisfy the upper-$\gamma$-extension property.
\end{description}
Indeed, if there exists such a $\gamma$, then all but at most $\eps |X|^{t-1} \cdot |X| + |X|^{t-1} \cdot \gamma |X| \leq 2 \eps |X|^t$ members of $X^t$ satisfy both the upper and lower bounds seen in~\eqref{eq:tuple} concluding the proof of the tuple lemma. In the sequel, we show that one may take $\gamma = \zeta$.

Let
$$
	\mathcal{T} = \left\{\boldsymbol{x} \in X^{t-1}\colon \; \text{$\boldsymbol{x}$ satisfies Properties~(E.1-2) with $\gamma = \zeta$}\right\}.
$$
We claim that
\begin{equation}\label{eq:tuples-with-1-2}
	|\mathcal{T}| \geq \left(1-\frac{\zeta}{2}\right)|X|^{t-1}.
\end{equation}
In order to prove~\eqref{eq:tuples-with-1-2}, we consider properties~(E.1) and~(E.2) separately. Starting with Property~(E.1), note that $P$ is a $(\delta_2,d_2)$-regular triad, where $(d_2 \pm \delta_2)^{t-1} = (1\pm \zeta)d_2^{t-1}$, owing to $\delta_2 \ll \zeta d_2^{t}$ which holds by~\eqref{eq:delta2-r}. Since, moreover, $\delta_2 \ll \zeta/t$ holds by~\eqref{eq:delta2-r}, it follows by Lemma~\ref{lem:tuple-graphs} that all but at most $2\delta_2(t-1)|X|^{t-1} \leq \frac{\zeta}{4}|X|^{t-1}$ of the members of $X^{t-1}$ satisfy Property~(E.1) with $\gamma = \zeta$. Next, it follows by the induction hypothesis, applied with $t_{\ref{lem:tuple}} = t-1$, $\eps_{\ref{lem:tuple}} = \zeta/8$, and $d_{\ref{lem:tuple}}^{(3)} = d_3$, that all but at most $\frac{\zeta}{4}|X|^{t-1}$ of the members of $X^{t-1}$ satisfy Property~(E.2) with $\gamma = \zeta$.

Having established~\eqref{eq:tuples-with-1-2}, it suffices to prove the following two claims.

\begin{claim}\label{clm:few-fail-E3}
	All but at most $\frac{\eps}{4}|X|^{t-1}$ of the members of $\mathcal{T}$ satisfy Property~{\em (E.3)} with $\gamma = \zeta$.
\end{claim}

\begin{claim}\label{clm:few-fail-E4}
	All but at most $\frac{\eps}{4}|X|^{t-1}$ of the members of $\mathcal{T}$ satisfy Property~{\em (E.4)} with $\gamma = \zeta$.
\end{claim}

Indeed, if Claims~\ref{clm:few-fail-E3} and~\ref{clm:few-fail-E4} both hold, then coupled with~\eqref{eq:tuples-with-1-2}, all but at most
\[
	\left(\frac{\eps}{4} + \frac{\eps}{4} + \frac{\zeta}{2}\right)|X|^{t-1}
	\overset{\eqref{eq:zeta}}{\leq}
	\eps |X|^{t-1}
\] of the members of $X^{t-1}$ satisfy both the upper and lower $\zeta$-extension properties as required.

Define
\begin{align*}
	\mathcal{B}_L & = \{\boldsymbol{x} \in \mathcal{T}\colon \boldsymbol{x} \; \text{fails to satisfy Property~(E.3) with $\gamma =\zeta$}\}\,, \\
	\mathcal{B}_U & = \{\boldsymbol{x} \in \mathcal{T}\colon \boldsymbol{x} \; \text{fails to satisfy Property~(E.4) with $\gamma =\zeta$}\}\,.
\end{align*}
The next two claims assert that if $\mathcal{B}_L$ or $\mathcal{B}_U$ exceed a certain size, then $r$ pairwise disjoint $(t-1)$-tuples can be found in these sets, having the property that their joint neighbourhoods along $P$ in both $Y$ and $Z$ are `small' (see~\eqref{eq:small-joint} below).
These tuples are then used in the proofs of Claims~\ref{clm:few-fail-E3} and~\ref{clm:few-fail-E4}, respectively, in order to construct {\it witnesses of irregularity} contradicting the premise of Proposition~\ref{lem:tuple}.

\begin{claim}\label{clm:many-fail-E3}
	If $|\mathcal{B}_L| \geq \frac{\eps}{4}|X|^{t-1}$, then there exists a collection of $r$ pairwise vertex  disjoint $(t-1)$-tuples, namely $\boldsymbol{x}_1,\ldots,\boldsymbol{x}_r \in \mathcal{B}_L$, such that
	\begin{equation}\label{eq:small-joint}
		\left|N_P(\boldsymbol{x}_i,Y) \cap N_P(\boldsymbol{x}_j,Y) \right| \leq 2d_2^{2(t-1)}|Y|\quad \text{and} \quad
		\left|N_P(\boldsymbol{x}_i,Z) \cap N_P(\boldsymbol{x}_j,Z) \right| \leq 2d_2^{2(t-1)}|Z|
	\end{equation}
	holds for every $1 \leq i < j \leq r$.
\end{claim}

\begin{proof}
	Define an auxiliary graph $\Gamma$, where $V(\Gamma) = \mathcal{B}_L$ and two $(t-1)$-tuples in $\mathcal{B}_L$ are adjacent provided that they are disjoint and satisfy~\eqref{eq:small-joint}. It suffices to prove that
	\begin{equation}\label{eq:Gamma-density}
		e(\Gamma) \geq \left(1-d_2^{2(t-1)}\right)\binom{|\mathcal{B}_L|}{2}\,.
	\end{equation}
	Indeed, if~\eqref{eq:Gamma-density} holds, then $\Gamma$ contains a complete subgraph of order at least $d_2^{-2(t-1)} \overset{\eqref{eq:delta2-r}}{\geq }r$, by Tur\'an's theorem~\cite{Turan}.

	It remains to prove~\eqref{eq:Gamma-density}. First, note that any given $(t-1)$-tuple shares a vertex with at most $O(|X|^{t-2})$ other $(t-1)$-tuples. Consequently, at most $O(|X|^{2t-3}) = o\left(|X|^{2(t-1)}\right) = o\left(|\mathcal{B}_L|^2\right)$ members of $\mathcal{B}_L^{(2)}$ are disqualified from forming an edge in $\Gamma$ due to not being disjoint. Second, note that since $\delta_2 \ll \eps^2 d_2^{2t}/t$ holds by~\eqref{eq:delta2-r} and~\eqref{eq:zeta}, and since $P$ forms a $(\delta_2,d_2)$-regular triad, it follows by Lemma~\ref{lem:tuple-graphs} that at most
	$$
		2\delta_2 \cdot 2(t-1)\cdot|X|^{2(t-1)} \ll \frac{\eps^2 \cdot d_2^{2(t-1)}}{2^{10}}|\mathcal{B}_L|^2 \leq \frac{d_2^{2(t-1)}}{2}\binom{|\mathcal{B}_L|}{2}
	$$
	of the $2(t-1)$-tuples $\boldsymbol{x} \in X^{2(t-1)}$ satisfy $|N_P(\boldsymbol{x},Y)| > 2 d_2^{2(t-1)}|Y|$ or $|N_P(\boldsymbol{x},Z)| > 2 d_2^{2(t-1)}|Z|$. Hence, at most $\frac{d_2^{2(t-1)}}{2}\binom{|\mathcal{B}_L|}{2}$ members of $\binom{\mathcal{B}_L}{2}$ are disjoint but fail to satisfy~\eqref{eq:small-joint}. These two observations complete the proof of~\eqref{eq:Gamma-density}.\phantom\qedhere\hfill{$\blacksquare$}
\end{proof}

Along the same lines one can prove the following claim.
\begin{claim}\label{clm:many-fail-E4}
	If $|\mathcal{B}_U| \geq \frac{\eps}{4}|X|^{t-1}$, then there exists a collection of $r$ pairwise vertex  disjoint $(t-1)$-tuples in $\mathcal{B}_U$ such that any two of them satisfy~\eqref{eq:small-joint}.\hfill{$\blacksquare$}
\end{claim}

We proceed to prove Claims~\ref{clm:few-fail-E3} and~\ref{clm:few-fail-E4}. Since our proofs of both claims are quite similar, we provide a detailed proof of  Claim~\ref{clm:few-fail-E3}, but for Claim~\ref{clm:few-fail-E4} we only account for the main differences between the two arguments.

\begin{proof}[Proof of Claim~\ref{clm:few-fail-E3}]
	Suppose for a contradiction that $|\mathcal{B}_L| \geq \frac{\eps}{4}|X|^{t-1}$ and let $\boldsymbol{x}_1,\ldots,\boldsymbol{x}_r$ be a collection of $r$ pairwise-disjoint members of $\mathcal{B}_L$ satisfying~\eqref{eq:small-joint}; the existence of such a collection is ensured by Claim~\ref{clm:many-fail-E3}. With each $\boldsymbol{x} \in \{\boldsymbol{x}_1,\ldots, \boldsymbol{x}_r \}$ we associate a subgraph $Q_{\boldsymbol{x}} \subseteq P$. We then prove that while
	$\left|\bigcup_{j=1}^r \mathcal{K}_3(Q_{\boldsymbol{x}_j})\right| \geq \delta_3 |\mathcal{K}_3(P)|$ holds, the collection $(Q_{\boldsymbol{x}_1},\ldots,Q_{\boldsymbol{x}_r})$ fails to satisfy~\eqref{eq:str-reg} (with the appropriate constants) and thus forms a {\em witness of irregularity} for $H$ with respect to $P$, contradicting the premise of the Tuple Lemma.

	For every $\boldsymbol{x} \in \{\boldsymbol{x}_1,\ldots, \boldsymbol{x}_r \}$, there exists a set $X_{\boldsymbol{x}} \subseteq X$ satisfying $|X_{\boldsymbol{x}}| = \zeta|X|$ (a quantity which we assume is integral) having the property that each of its members fails to satisfy~\eqref{eq:E3} with $\gamma = \zeta$ if added to $\boldsymbol{x}$ so as to form a $t$-tuple. Since $\delta_2 \ll \zeta$ holds by~\eqref{eq:delta2-r}, it follows that $|X_{\boldsymbol{x}}| \geq \delta_2 |X|$. Since $\mathcal{B}_L \subseteq \mathcal{T}$, the tuple $\boldsymbol{x}$ satisfies Property~(E.1) with $\gamma = \zeta$. Since, moreover, $\zeta \leq 1/2$ by~\eqref{eq:zeta}, and $\delta_2 \ll d_2^{t-1}$ by~\eqref{eq:delta2-r}, it follows that
	$$
		|N_P(\boldsymbol{x},Y)| \geq (1-\zeta)d_2^{t-1}|Y| \geq \delta_2|Y|\quad \text{and}\quad |N_P(\boldsymbol{x},Z)| \geq (1-\zeta)d_2^{t-1}|Z| \geq \delta_2|Z|
	$$
	These lower bounds on the cardinalities of $X_{\boldsymbol{x}}$, $N_P(\boldsymbol{x},Y)$, and $N_P(\boldsymbol{x},Z)$ together with Lemma~\ref{lem:slicing} collectively imply that $P[X_{\boldsymbol{x}}, N_P(\boldsymbol{x},Y)]$ and $P[X_{\boldsymbol{x}}, N_P(\boldsymbol{x},Z)]$ are both $(\sqrt{\delta_2},d_2 \pm \delta_2)$-regular; indeed
	$$
		\max\left\{\frac{\delta_2}{\zeta},\; \frac{\delta_2}{(1-\zeta)d_2^{t-1}},\; 2\delta_2\right\} \ll \sqrt{\delta_2}
	$$
	follows, as $\delta_2 \ll \zeta^2 d_2^{2(t-1)}$ holds by~\eqref{eq:delta2-r}, and as $\zeta \leq 1- \zeta$ due to $\zeta \ll 1/2$ holding by~\eqref{eq:zeta}. Lemma~\ref{lem:two-sided-tri-cnt} implies that
	\begin{equation}\label{eq:K3s-Qx}
		|\mathcal{K}_3(Q_{\boldsymbol{x}})| = \left(1\pm \zeta\right) \zeta\left(d_3^{t-1} \pm \zeta \right) d_2^{2t+1}|X||Y|Z|,
	\end{equation}
	where
	\begin{equation}\label{eq:Qx}
		Q_{\boldsymbol{x}} = P[X_{\boldsymbol{x}},N_P(\boldsymbol{x},Y)] \discup P[X_{\boldsymbol{x}}, N_P(\boldsymbol{x},Z)] \discup L_H({\boldsymbol{x}},P) \subseteq P.
	\end{equation}
	To see why the upper bound stipulated in~\eqref{eq:K3s-Qx} holds, note that Lemma~\ref{lem:two-sided-tri-cnt} applied to $Q_{\boldsymbol{x}}$ yields
	\begin{align*}
		|\mathcal{K}_3(Q_{\boldsymbol{x}})| & \leq (d_2 +\delta_2 +\sqrt{\delta_2})(d_2 +\delta_2)|X_{\boldsymbol{x}}|e(L_H(\boldsymbol{x},P)) + 2\sqrt{\delta_2}|X_{\boldsymbol{x}}||N_P(\boldsymbol{x},Y)||N_P(\boldsymbol{x},Z)| \\
		                                    & \leq d_2^2|X_{\boldsymbol{x}}|e(L_H(\boldsymbol{x},P)) + 5\sqrt{\delta_2} d_2 |X_{\boldsymbol{x}}| e(L_H(\boldsymbol{x},P))+ 2\sqrt{\delta_2}|X_{\boldsymbol{x}}||Y||Z|.
	\end{align*}
	where the last inequality is supported by $\delta_2 \ll d_2$ which holds due to~\eqref{eq:delta2-r}.

	Since $\mathcal{B}_L \subseteq \mathcal{T}$, the tuple $\boldsymbol{x}$ satisfies Property~(E.2) with $\gamma = \zeta$, implying that
	$$
		e(L_H(\boldsymbol{x},P)) \leq \left(d_3^{t-1} + \zeta\right)d_2^{2t-1}|Y||Z|.
	$$
	Therefore,
	\begin{align*}
		|\mathcal{K}_3(Q_{\boldsymbol{x}})| & \leq d_2^2 (d_3^{t-1} +\zeta)d_2^{2t-1}|X_{\boldsymbol{x}}||Y||Z| + 7\sqrt{\delta_2} |X_{\boldsymbol{x}}||Y||Z| \\
		                                    & \leq \left(1 + \zeta\right)(d_3^{t-1} +\zeta)d_2^{2t+1}|X_{\boldsymbol{x}}||Y||Z|                               \\
		                                    & = \left(1 + \zeta\right)\zeta(d_3^{t-1} + \zeta) d_2^{2t+1}|X||Y||Z|\,,
	\end{align*}
	where for the last inequality we rely on $\delta_2 \ll \zeta^2(d_3^{t-1} +\zeta)^2d_2^{4t+2}$, which holds by~\eqref{eq:delta2-r}.

	A similar argument delivers the lower bound seen in~\eqref{eq:K3s-Qx}. Indeed, Lemma~\ref{lem:two-sided-tri-cnt} applied to~$Q_{\boldsymbol{x}}$ also provides
	\begin{align*}
		|\mathcal{K}_3(Q_{\boldsymbol{x}})| & \geq (d_2 -\delta_2 - \sqrt{\delta_2})(d_2-\delta_2)|X_{\boldsymbol{x}}| e(L_H(\boldsymbol{x},P)) - 2 \sqrt{\delta_2} |X_{\boldsymbol{x}}| |Y||Z| \\
		                                    & \geq \left(d_3^{t-1}-\zeta \right)d_2^{2t+1} |X_{\boldsymbol{x}}| |Y||Z| - 7 \sqrt{\delta_2}|X_{\boldsymbol{x}}||Y||Z|                            \\
		                                    & \geq  \left(1-\zeta \right)\zeta\left(d_3^{t-1}-\zeta \right)d_2^{2t+1}|X||Y||Z|\,,
	\end{align*}
	where the penultimate inequality relies on
	$$
		e(L_H(\boldsymbol{x},P)) \geq \left(d_3^{t-1} - \zeta\right)d_2^{2t-1}|Y||Z|\,,
	$$
	which holds owing to $\boldsymbol{x}$ satisfying Property~(E.2) with $\gamma = \zeta$, and the last inequality is owing to $\delta_2 \ll \zeta^2(d_3^{t-1} -\zeta)^2d_2^{4t+2}$ which holds by~\eqref{eq:delta2-r}. This completes the proof of~\eqref{eq:K3s-Qx}.

	We proceed to prove that the collection $(Q_{\boldsymbol{x}_1},\ldots,Q_{\boldsymbol{x}_r})$ satisfies $\left| \bigcup_{j=1}^r \mathcal{K}_3(Q_{\boldsymbol{x}_j}) \right| \geq \delta_3 |\mathcal{K}_3(P)|$. Note that

	\begin{align}
		\left|\bigcup_{j=1}^r \mathcal{K}_3(Q_{\boldsymbol{x}_j}) \right| & \overset{\phantom{\eqref{eq:K3s-Qx}}}{\geq} \sum_{j=1}^r |\mathcal{K}_3(Q_{\boldsymbol{x}_j})| - \sum_{1 \leq i < j \leq r} |\mathcal{K}_3(Q_{\boldsymbol{x}_i}) \cap \mathcal{K}_3(Q_{\boldsymbol{x}_j})| \nonumber                                           \\
		                                                                  & \overset{\eqref{eq:K3s-Qx}}{\geq} r\left(1-\zeta\right)\zeta\left(d_3^{t-1} -\zeta\right)d_2^{2t+1}|X||Y||Z| - \sum_{1 \leq i < j \leq r} \left|\mathcal{K}_3(Q_{\boldsymbol{x}_i}) \cap \mathcal{K}_3(Q_{\boldsymbol{x}_j}) \right|\,. \label{eq:witness-union}
	\end{align}
	Given indices $1 \leq i < j \leq r$, a triangle found in $\mathcal{K}_3(Q_{\boldsymbol{x}_i}) \cap \mathcal{K}_3(Q_{\boldsymbol{x}_j})$ has its vertices residing in the sets $X_{\boldsymbol{x}_i} \cap X_{\boldsymbol{x}_j}$, $N_P(\boldsymbol{x}_i,Y) \cap N_P(\boldsymbol{x}_j,Y)$, and $N_P(\boldsymbol{x}_i,Z) \cap N_P(\boldsymbol{x}_j,Z)$. If any of these sets has size at most a $\sqrt{\delta_2}$-fraction of its respective host, namely $X,Y,Z$, respectively, then
	\begin{equation}\label{eq:cap-small}
		\left|\mathcal{K}_3(Q_{\boldsymbol{x}_i}) \cap \mathcal{K}_3(Q_{\boldsymbol{x}_j})\right| \leq \sqrt{\delta_2} |X||Y||Z|\,.
	\end{equation}
	In the complementary case, the subgraphs
	\[
		P[X_{\boldsymbol{x}_i} \cap X_{\boldsymbol{x}_j}, N_P(\boldsymbol{x}_i,Y) \cap N_P(\boldsymbol{x}_j,Y)]\,,
		\quad
		P[X_{\boldsymbol{x}_i} \cap X_{\boldsymbol{x}_j}, N_P(\boldsymbol{x}_i,Z)\cap N_P(\boldsymbol{x}_j,Z)]\,,
	\]
	and
	\[
		P[N_P(\boldsymbol{x}_i,Y) \cap N_P(\boldsymbol{x}_j,Y), N_P(\boldsymbol{x}_i, Z) \cap N_P(\boldsymbol{x}_j, Z)]
	\]
	are all $(\sqrt{\delta_2},d_2 \pm \delta_2)$-regular, by Lemma~\ref{lem:slicing}. Hence,
	\begin{multline*}
		\left|\mathcal{K}_3(Q_{\boldsymbol{x}_i}) \cap \mathcal{K}_3(Q_{\boldsymbol{x}_j})\right| \\
		\leq
		\left((d_2+\delta_2)^3 + 4\sqrt{\delta_2}\right)|X_{\boldsymbol{x}_i} \cap X_{\boldsymbol{x}_j}|
		|N_P(\boldsymbol{x}_i,Y) \cap N_P(\boldsymbol{x}_j,Y)||N_P(\boldsymbol{x}_i,Z) \cap N_P(\boldsymbol{x}_j,Z)|\,.
	\end{multline*}
	Consequently, we have
	\begin{align}
		\left|\mathcal{K}_3(Q_{\boldsymbol{x}_i}) \cap \mathcal{K}_3(Q_{\boldsymbol{x}_j})\right|
		 & \leq \left(d_2^3 + 3 \delta_2d_2 + 4\sqrt{\delta_2}\right) \cdot \zeta |X| \cdot 2d_2^{2(t-1)}|Y| \cdot 2d_2^{2(t-1)}|Z| \nonumber \\
		 & \leq 8\zeta d_2^{4t-1} |X||Y||Z| \label{eq:cap-main-term}\end{align}
	holds, where the first inequality holds by~\eqref{eq:tri-cnt}, the second inequality is supported by~\eqref{eq:small-joint} which is satisfied by $\boldsymbol{x}_i$ and $\boldsymbol{x}_j$, and for the third inequality, we rely on $\delta_2 \ll d_2^6$ supported by~\eqref{eq:delta2-r}.  Then,
	\begin{align}\label{eq:all-intersections}
		\sum_{1 \leq i < j \leq r} \left|\mathcal{K}_3(Q_{\boldsymbol{x}_i}) \cap \mathcal{K}_3(Q_{\boldsymbol{x}_j})\right| 
		& \overset{\eqref{eq:cap-small}, \eqref{eq:cap-main-term} }{\leq} 
		r^2\left(\sqrt{\delta_2}+ 8 \zeta d_2^{4t-1}\right)|X||Y||Z| \nonumber \\                                                                                                                     		& \overset{\phantom{\eqref{eq:cap-small}, \eqref{eq:cap-main-term}}}{\leq} 
		9 r^2 \zeta d_2^{4t-1} |X||Y||Z|,
	\end{align}
	where the last inequality is owing to $\delta_2 \ll \zeta^2 d_2^{8t-2}$, supported by~\eqref{eq:delta2-r}.
	Substituting this last estimate into~\eqref{eq:witness-union}, we arrive at
	\begin{align}
		\bigg|\bigcup_{j=1}^r \mathcal{K}_3(Q_{\boldsymbol{x}_j}) \bigg| & \overset{\phantom{\eqref{eq:delta2-r}}}{\geq} r\left(1-\zeta\right)\zeta\left(d_3^{t-1} -\zeta\right)d_2^{2t+1}|X||Y||Z| - 9r^2\zeta d_2^{4t-1}|X||Y||Z| \nonumber                    \\
		                                                                  & \overset{\phantom{\eqref{eq:delta2-r}}}{\geq} \left(r \zeta d_3^{t-1}d_2^{2t+1} - r \zeta^2d_3^{t-1}d_2^{2t+1} - r\zeta^2d_2^{2t+1} -9r^2\zeta d_2^{4t-1} \right) |X||Y||Z| \nonumber \\
		                                                                  & \overset{\eqref{eq:delta2-r}}{=} \left(d_3^{t-1} - \zeta d_3^{t-1} - \zeta - 9\zeta \right)\zeta^2 d_2^3 |X||Y||Z| \nonumber                                                          \\
		                                                                  & \overset{\phantom{\eqref{eq:delta2-r}}}{\geq}  (d_3^{t-1} - 11\zeta) \zeta^2 d_2^{3}|X||Y||Z| \label{eq:lower-bound-witness'}                                                         \\
		                                                                  & \overset{\phantom{\eqref{eq:delta2-r}}}{\geq} 2\delta_3 d_2^3 |X||Y||Z| \nonumber                                                                                                     \\
		                                                                  & \overset{\phantom{\eqref{eq:delta2-r}}}{\geq} \delta_3 |\mathcal{K}_3(P)|. \label{eq:witness-volume}
	\end{align}
	The first inequality above holds by~\eqref{eq:K3s-Qx} and~\eqref{eq:all-intersections}. For the penultimate inequality, we rely on $\delta_3 \ll (d_3^{t-1} - 11\zeta) \zeta^2$, which holds by~\eqref{eq:delta3} (as well as~\eqref{eq:zeta} asserting that $\zeta \ll d_3^{t-1}$). For the last inequality, note that
	$\delta_2 \ll d_2^3$, which is supported by~\eqref{eq:delta2-r}, and~\eqref{eq:tri-cnt} collectively yield
	$$
		|\mathcal{K}_3(P)| \leq (d_2^3+4\delta_2)|X||Y||Z| \leq 2 d_2^3 |X||Y||Z|.
	$$

	Gearing up towards examining~\eqref{eq:str-reg} and proving that it fails to hold for $(Q_{\boldsymbol{x}_1},\ldots,Q_{\boldsymbol{x}_r})$, note that, in particular,
	\begin{equation}\label{eq:lower-bound-witness}
		d_3\bigg|\bigcup_{j=1}^r \mathcal{K}_3(Q_{\boldsymbol{x}_j}) \bigg| \overset{\eqref{eq:lower-bound-witness'}}{\geq} (d_3^t -11d_3\zeta) \zeta^2 d_2^3 |X||Y||Z| \geq (d_3^t -11\zeta)\zeta^2 d_2^3 |X||Y||Z|
	\end{equation}
	holds. On the other hand,
	\begin{align}
		\bigg| \bigcup_{j=1}^r \left(E_H\cap \mathcal{K}_3(Q_{\boldsymbol{x}_j}) \right) \bigg| & \overset{\phantom{\eqref{eq:delta2-r}}}{\leq} \sum_{j=1}^r \left|E_H\cap \mathcal{K}_3(Q_{\boldsymbol{x}_j}) \right| \nonumber \\
		                                                                                         & \overset{\phantom{\eqref{eq:delta2-r}}}{<} r \zeta |X| \cdot (d_3^t-12\zeta)d_2^{2t+1} |Y||Z| \nonumber                        \\
		                                                                                         & \overset{\eqref{eq:delta2-r}}{=} (d_3^t-12\zeta)\zeta^2 d_2^3|X||Y||Z|\,, \label{eq:upper-bound-witness}
	\end{align}
	where the second inequality holds since, by definition, for every $j \in [r]$, every member of $X_{\boldsymbol{x}_j}$ fails to satisfy~\eqref{eq:E3} with $\gamma = \zeta$ if added to $\boldsymbol{x}_j$ so as to form a $t$-tuple.

	Therefore,
	\begin{align}
		\bigg| \Big| \bigcup_{j=1}^r \left(E_H\cap \mathcal{K}_3(Q_{\boldsymbol{x}}) \right) \Big| - d_3\Big|\bigcup_{j=1}^r \mathcal{K}_3(Q_{\boldsymbol{x}_j}) \Big|\bigg| & \overset{\eqref{eq:lower-bound-witness}, \eqref{eq:upper-bound-witness}}{\geq}
		\left| d_3^t - 12\zeta - d_3^t + 11\zeta\right|\zeta^2 d_2^3|X||Y||Z| \nonumber                                                                                                                                                                                                                             \\
		                                                                                                                                                                      & \overset{\phantom{\eqref{eq:lower-bound-witness}, \eqref{eq:upper-bound-witness}}}{=} \zeta^3 d_2^3 |X||Y||Z| \nonumber              \\
		                                                                                                                                                                      & \overset{\phantom{\eqref{eq:lower-bound-witness}, \eqref{eq:upper-bound-witness}}}{\gg} 2 \delta_3 d_2^3 |X||Y||Z| \nonumber         \\
		                                                                                                                                                                      & \overset{\phantom{\eqref{eq:lower-bound-witness}, \eqref{eq:upper-bound-witness}}}{\geq} \delta_3 |\mathcal{K}_3(P)|\,. \label{eq:gap}
	\end{align}
	For the penultimate inequality, we rely on $\delta_3 \ll \zeta^3$ which holds by~\eqref{eq:delta3}, and for the last inequality we rely, once more, on $|\mathcal{K}_3(P)| \leq 2d_2^3 |X||Y||Z|$ which holds by~\eqref{eq:tri-cnt}.

	To conclude the proof of Claim~\ref{clm:few-fail-E3} and thus the lower bound seen in~\eqref{eq:tuple}, note that~\eqref{eq:witness-volume} and~\eqref{eq:gap} contradict the assumption that $H$ is $(\delta_3,d_3,r)$-regular with respect to $P$.\phantom\qedhere\hfill{$\blacksquare$}
\end{proof}

We proceed to the proof of Claim~\ref{clm:few-fail-E4} which supports the upper bound seen in~\eqref{eq:tuple}. As mentioned above, our proof of Claim~\ref{clm:few-fail-E4} is quite similar to that seen for Claim~\ref{clm:few-fail-E3}; hence, we do not give a detailed proof of Claim~\ref{clm:few-fail-E4}, but rather specify the main points where our arguments diverge from their counterparts appearing in the proof of Claim~\ref{clm:few-fail-E3}.

\begin{proof}[Proof of Claim~\ref{clm:few-fail-E4}]
	Suppose for a contradiction that $|\mathcal{B}_U| \geq \frac{\eps}{4}|X|^{t-1}$ and let $\boldsymbol{x}_1,\ldots,\boldsymbol{x}_r$ be a collection of $r$ pairwise-disjoint members of $\mathcal{B}_U$ satisfying~\eqref{eq:small-joint}; the existence of such a collection is ensured by Claim~\ref{clm:many-fail-E4}. For every $\boldsymbol{x} \in \{\boldsymbol{x}_1,\ldots, \boldsymbol{x}_r \}$ there exists a set $X_{\boldsymbol{x}} \subseteq X$ of size $|X_{\boldsymbol{x}}| = \zeta|X|$, having the property that each of its members fails to satisfy~\eqref{eq:E4} with $\gamma = \zeta$ if added to $\boldsymbol{x}$ so as to form a $t$-tuple.

	Define $Q_{\boldsymbol{x}}$ as seen in~\eqref{eq:Qx}. As in the proof of Claim~\ref{clm:few-fail-E3}, we prove that the collection $(Q_{\boldsymbol{x}_1},\ldots,Q_{\boldsymbol{x}_r})$ satisfies $\big| \bigcup_{j=1}^r \mathcal{K}_3(Q_{\boldsymbol{x}_j}) \big| \geq \delta_3 |\mathcal{K}_3(P)|$ but fails to satisfy~\eqref{eq:str-reg} (with the appropriate constants). The estimates~\eqref{eq:K3s-Qx},~\eqref{eq:cap-main-term},~\eqref{eq:all-intersections}, and consequently~\eqref{eq:witness-volume} all hold in the setting of the current claim; for indeed, these are all established using properties~(E.1-2) (and the cardinality of $X_{\boldsymbol{x}}$) solely and without any reference to the violation of Property~(E.3) (which in the current claim we do not assume). These estimates support $\big| \bigcup_{j=1}^r \mathcal{K}_3(Q_{\boldsymbol{x}_j}) \big| \geq \delta_3 |\mathcal{K}_3(P)|$ holding in the setting of Claim~\ref{clm:few-fail-E4} as well.

	To see that~\eqref{eq:str-reg} is violated in the current setting, note that
	\begin{align}
		d_3\bigg| \bigcup_{j=1}^r \mathcal{K}_3(Q_{\boldsymbol{x}_j}) \bigg| & \leq d_3 \sum_{j=1}^r |\mathcal{K}_3(Q_{\boldsymbol{x}_j})| \nonumber                                  \\
		                                                                      & \overset{\eqref{eq:K3s-Qx}}{\leq} d_3r(1+\zeta)\zeta(d_3^{t-1} + \zeta) d_2^{2t+1} |X||Y||Z| \nonumber \\
		                                                                      & \overset{\eqref{eq:delta2-r}}{=} d_3 (1+\zeta)(d_3^{t-1}+\zeta) \zeta^2 d_2^3 |X||Y||Z| \nonumber      \\
		                                                                      & = (d_3^t + d_3\zeta + d_3^t\zeta + d_3 \zeta^2)\zeta^2d_2^3|X||Y||Z| \nonumber                         \\
		                                                                      & \leq (d_3^t + 3\zeta)\zeta^2 d_2^3 |X||Y||Z|. \label{eq:witness-upper2}
	\end{align}
	Additionally, note that
	\begin{align}
		\bigg|\bigcup_{j=1}^r (E_H \cap \mathcal{K}_3(Q_{\boldsymbol{x}_j}) \bigg| & \geq \sum_{j=1}^r |E_H \cap \mathcal{K}_3(Q_{\boldsymbol{x}_j})| - \sum_{1 \leq i < j \leq r} |\mathcal{K}_3(Q_{\boldsymbol{x}_i}) \cap \mathcal{K}_3(Q_{\boldsymbol{x}_i})| \nonumber \\
		                                                                            & \geq r\zeta |X| \cdot (d_3^t + 13\zeta)d_2^{2t+1}|Y||Z| - 9r^2\zeta d_2^{4t-1} |X||Y||Z| \nonumber                                                                                     \\
		                                                                            & = (d_3^t + 13\zeta - 9\zeta)\zeta^2 d_2^3 |X||Y||Z| \nonumber                                                                                                                          \\
		                                                                            & = (d_3^t + 4\zeta)\zeta^2 d_2^3 |X||Y||Z|. \label{eq:witness-lower2}
	\end{align}
	Combining~\eqref{eq:witness-upper2} and~\eqref{eq:witness-lower2} yields
	\begin{align*}
		\bigg| \Big| \bigcup_{j=1}^r \left(E_H\cap \mathcal{K}_3(Q_{\boldsymbol{x}}) \right) \Big| - d_3\Big|\bigcup_{j=1}^r \mathcal{K}_3(Q_{\boldsymbol{x}_j}) \Big|\bigg| & \geq |d_3^t + 4\zeta - d_3^t - 3\zeta| \zeta^2 d_2^3|X||Y||Z| \\
		                                                                                                                                                                      & = \zeta^3 d_2^3 |X||Y||Z|                                     \\
		                                                                                                                                                                      & \gg 2\delta_3 d_2^3 |X||Y||Z|                                 \\
		                                                                                                                                                                      & \geq \delta_3 |\mathcal{K}_3(P)|,
	\end{align*}
	establishing the violation of~\eqref{eq:str-reg} and concluding the proof.\phantom\qedhere\hfill{$\blacksquare$}
\end{proof}

\end{document}